\documentclass[a4paper,10pt]{amsart}

\usepackage[T1]{fontenc}

\usepackage[utf8]{inputenc}

\usepackage{lmodern}

\usepackage[english]{babel}

\usepackage{amsmath}

\usepackage{amssymb}

\usepackage{mathtools}

\usepackage{mathrsfs}

\usepackage{bookmark}

\usepackage{hyperref} 
\hypersetup{
	colorlinks=true,
	linkcolor=blue,
	filecolor=magenta,      
	urlcolor=cyan,
}

\usepackage[alphabetic]{amsrefs}

\usepackage{enumitem}
\usepackage{graphicx}
\usepackage{tikz}
\usepackage{tikz-cd}
\usetikzlibrary{matrix,arrows,decorations.pathmorphing}
\usepackage[all,cmtip]{xy}

\newtheorem*{thm-no-num}{Theorem}
\newtheorem*{df-no-num}{Definition}
\newtheorem*{kl-no-num}{Key Lemma}
\newtheorem{thm}{Theorem} [section]
\newtheorem{prop}[thm]{Proposition} 
\newtheorem{lm}[thm]{Lemma} 
\newtheorem{kl}[thm]{Key Lemma}
\newtheorem{cor}[thm]{Corollary} 
\theoremstyle{remark}
\newtheorem{rmk}[thm]{Remark}
\newtheorem*{rmk-no-num}{Remark}

\theoremstyle{definition} 
\newtheorem {df}[thm]{Definition}

\newcommand{\bA}{\mathbb{A}}
\newcommand{\PP}{\mathbb{P}}
\newcommand{\ZZ}{\mathbb{Z}}
\newcommand{\FF}{\mathbb{F}}
\newcommand{\OO}{\mathcal{O}}
\newcommand{\cl}[1]{\mathcal{#1}}
\newcommand*{\sheafhom}{\mathcal{H}\kern -.5pt om}

\newcommand{\Dcal}{\mathcal{D}}
\newcommand{\Ecal}{\cl{E}}
\newcommand{\Fcal}{\cl{F}}
\newcommand{\Gcal}{\cl{G}}
\newcommand{\Bcal}{\cl{B}}

\newcommand{\Zcal}{\cl{Z}}
\newcommand{\Hcal}{\cl{H}}
\newcommand{\Ical}{\cl{I}}
\newcommand{\Lcal}{\cl{L}}
\newcommand{\Xcal}{\cl{X}}
\newcommand{\Ycal}{\cl{Y}}
\newcommand{\Vcal}{\cl{V}}
\newcommand{\Wcal}{\cl{W}}

\newcommand{\Gm}{\mathbb{G}_m}
\newcommand{\GLt}{\textnormal{GL}_3}
\newcommand{\PGLt}{\textnormal{PGL}_2}

\newcommand{\im}{{\rm im}}
\newcommand{\pr}{{\rm pr}}

\newcommand{\sng}{^{\rm sing}}

\newcommand{\wt}[1]{\widetilde{#1}}
\newcommand{\wh}[1]{\widehat{#1}}
\newcommand{\ol}[1]{\overline{#1}}

\newcommand{\Apgl}{A_{\textnormal{PGL}_2}}
\newcommand{\Agl}{A_{\textnormal{GL}_3}}

\newcommand{\CHgl}{CH_{\textnormal{GL}_3}}

\newcommand{\Het}{{\rm H}^{\bullet}}
\newcommand{\Inv}{{\rm Inv}^{\bullet}}
\newcommand{\Spec}{{\rm Spec}}

\newcommand{\del}{\partial}
\begin{document}
	\title[Cohomological invariants of hyperelliptic curves of odd genus]{Cohomological invariants of the stack of hyperelliptic curves of odd genus}
	\author[A. Di Lorenzo]{Andrea Di Lorenzo}
	\address{Scuola Normale Superiore, Piazza dei Cavalieri 7, 56126 Pisa, Italy}
	\email{andrea.dilorenzo@sns.it}
	\date{\today}
\begin{abstract}
	We compute the cohomological invariants of $\Hcal_g$, the moduli stack of smooth hyperelliptic curves, for every odd $g$.
\end{abstract}
\maketitle
\tableofcontents
\section*{Introduction}
Characteristic classes are important topological invariants. They were first introduced by E. L. Stiefel and H. Whitney as a tool to study vector bundles on manifolds, and they were later generalized to principal bundles.

Fix a topological group $G$ and a cohomology theory $H$: then characteristic classes are a functorial way to associate to every principal $G$-bundle over a topological space $X$ a cohomology class in $H(X)$. In other terms, characteristic classes are natural transformations from the functor
$$ \Bcal G: {\rm Top}\longrightarrow {\rm Set},\quad X\longmapsto\left\{G\text{-bundles over }X\right\} $$
to the cohomology functor $X\mapsto H(X)$.

Cohomological invariants first appeared as a reformulation of this idea in an algebraic setting. More precisely, fix a positive number $p$, a field $k_0$ whose characteristic does not divide $p$ and an algebraic group $G$. Then we replace the category ${\rm Top}$ with ${\rm Field}/k_0$, the category of field extensions of $k_0$, and the cohomology theory $H$ with
$$\Het:{\rm Field}/k_0\longrightarrow {\rm Ring},\quad K\longmapsto \oplus_i H^i_{\'{e}t}(\Spec(K),\mu_p^{\otimes i}) $$
where $\mu_p$ denotes the group of $p^{\rm th}$-roots of unity.
We also substitute $\Bcal G$ with its algebro-geometrical counterpart, namely the classifying stack $\Bcal G$ or, better, its functor of points:
$$P_{\Bcal G}:{\rm Field}/k_0\longrightarrow {\rm Set},\quad K\longmapsto \left\{G\text{-torsors over }\Spec(K)\right\}   $$
Cohomological invariants are then defined by copying the definition of characteristic classes in topology:
\begin{df-no-num}[\cite{GMS}*{Def. 1.1}]\label{def:coh inv BG}
	A cohomological invariant of $\Bcal G$ with coefficients in $\FF_p$ is a natural transformation of functors $$P_{\Bcal G} \longrightarrow\Het:=\oplus_i H^i_{\'{e}t}(-,\mu_p^{\otimes i})$$ The graded-commutative ring of cohomological invariants of $\Bcal G$ is denoted $\Inv(\Bcal G)$.
\end{df-no-num} 

The first appearance of cohomological invariants, though not in this formulation, can be traced back to the seminal paper of Witt \cite{Wit} and since that they have been extensively studied (see \cite{GMS}).

In the recent work \cite{PirAlgStack}, Roberto Pirisi extended the notion of cohomological invariants from classifying stacks to smooth algebraic stacks over $k_0$:
\begin{df-no-num}[\cite{PirAlgStack}*{Def. 1.1}]
	Let $\Xcal$ be a smooth algebraic stack over $k_0$. Then a cohomological invariant of $\Xcal$ is a natural transformation 
	$$ P_{\Xcal}\longrightarrow \Het $$
	from the functor of points of $\Xcal$ to $\Het$ which satisfies a certain continuity condition (see \cite{PirAlgStack}*{Def. 1.1}).
	
	The graded-commutative ring of cohomological invariants of a smooth algebraic stack $\Xcal$ is denoted $\Inv(\Xcal)$.
\end{df-no-num} 
The main result of the present work is the following:
\begin{thm-no-num}
	Let $k_0$ be an algebraically closed field of characteristic $\neq 2$, and let $\Hcal_g$ denote the moduli stack of smooth hyperelliptic curves of odd genus $g\geq 3$ over $k_0$.
	
	Then the graded-commutative ring of cohomological invariants $\Inv(\Hcal_g)$ with coefficients in $\FF_2$ is generated, as a graded $\FF_2$-vector space, by classes $$1,x_1,w_2,x_2,...,x_{g+1},x_{g+2}$$ where the degree of each $x_i$ is $i$ and $w_2$ is the second  Stiefel-Whitney class coming from $\Inv(\cl{B}\PGLt)$.
\end{thm-no-num}

\begin{rmk-no-num}
	The cohomological invariants $\Inv(\Hcal_g)$ with coefficients in $\FF_p$, when $g$ is even or equal to three, has already been computed by Pirisi in \cite{PirCohHypEven} and \cite{PirCohHypThree} when the base field is algebraically closed and its characteristic does not divide $p$.
	
	When $g$ is odd and $p\neq 2$, the generators of $\Inv(\Hcal_g)$ can also be deduced from the main result of \cite{PirCohHypThree}. The importance of the case $p=2$ is due to the fact that under this assumption, the cohomological invariants of $\Hcal_g$ present a richer structure: why this happens is currently a work in progress of Pirisi together with the author.	
\end{rmk-no-num} 
The computation of cohomological invariants is based on the isomorphism between $\Inv([X/G])$, where $X$ is a smooth scheme endowed with an action of an algebraic group $G$, and the equivariant Chow group with coefficients $A_G^0(X,\Het)$ (see \cite[section 4]{PirAlgStack}). The Chow groups with coefficients were first introduced by Rost in \cite{Rost} as a generalization of ordinary Chow groups. Their usage as a tool for computing cohomological invariants motivates the hypotheses that we made on the characteristic of the base field (see remark \ref{rmk:assumptions}).

The main obstruction to extend the computations on $\Hcal_3$ contained in \cite{PirCohHypThree} to the stacks $\Hcal_g$, where $g$ is any odd integer, consists in proving that a certain morphism of $\PGLt$-equivariant Chow groups with coefficients is zero.

More precisely, let $\PP(1,2n)$ denote the projective space of binary forms of degree $2n$, endowed with the ${\rm GL}_2$-action
$$ A\cdot f(x,y):=\det(A)^n f(A^{-1}(x,y)) $$
This action descends to a well defined action of $\PGLt$ on the same scheme. Let $\Delta_{1,2n}\subset\PP(1,2n)$ be the closed, $\PGLt$-invariant subscheme parametrising singular forms.
Then, to extend the results of Pirisi on $\Inv(\Hcal_3)$, it is enough to prove the following:
\begin{kl-no-num}
    Let $n\geq 1$ be an integer and let $k_0$ be an algebraically closed field of characteristic $\neq 2$. Let $i:\Delta_{1,2n}\hookrightarrow\PP(1,2n)$ be the inclusion of the subscheme of singular forms into the projective space of binary forms of degree $2n$ over $k_0$. Then the pushforward homomorphism:
	\[i_*: \Apgl^0(\Delta_{1,2n},\Het)\longrightarrow\Apgl^1(\PP(1,2n),\Het)\]
	between equivariant Chow groups with coefficients in $\Het:=\oplus_i H^i_{\'{e}t}(-,\mu_2^{\otimes i})$ vanishes.
\end{kl-no-num}
What enables us to to prove the key lemma is the notion of $\GLt$-counterpart of a $\PGLt$-scheme:
\begin{df-no-num}\cite{Dil}*{def. 1.1}
	Let $X$ be a scheme of finite type over a field, endowed with a $\PGLt$-action. Then a $\GLt$-counterpart of $X$ is a scheme $Y$ endowed with an action of $\GLt$ such that $[Y/\GLt]\simeq [X/\PGLt]$.
\end{df-no-num}
If $Y$ is a $\GLt$-counterpart of a $\PGLt$-scheme, it follows almost immediately (see \cite{Dil}*{prop. 1.3}) that
$$ \Apgl(X,\Het)\simeq\Agl(Y,\Het) $$
To prove the key lemma we first find the $\GLt$-counterparts of the $\PGLt$-schemes $\PP(1,2n)$ and $\Delta_{1,2n}$, which we call respectively $\PP(V_n)_3$ and $D_3$. We then apply a new argument, not available in the $\PGLt$-equivariant setting, to show that the morphism
$$ i_*:\Agl^0(D_3,\Het)\longrightarrow\Agl^1(\PP(V_n)_3,\Het)$$
vanishes.
\subsection*{Structure of the paper}
In section \ref{sec:equi Rost Chow} we recall some basic properties of equivariant Chow groups with coefficients. In section \ref{sec:coh inv} we prove the main theorem of the paper, assuming for the moment the key lemma \ref{lm:key}, whose proof is postponed to section \ref{sec:main lemma}. The strategy of proof is similar to the one contained in \cite{PirCohHypThree}. The remainder of the paper is devoted to develop the theory necessary to prove the key lemma \ref{lm:key}.

In section \ref{sec:GL3 counterpart} we introduce the notion of $\GLt$-counterpart of a $\PGLt$-scheme. 

In section \ref{sec:on the divisor} we study the geometry of a certain divisor $D$ contained in a projective bundle over an open subscheme of the affine space of quadratic ternary forms.

The observations made in this section are then applied in section \ref{sec:Chow comp} in order to do some intersection theoretical computations useful to prove the key lemma \ref{lm:key}: the proof is completed in section \ref{sec:main lemma}. 

For the convenience of the reader, a more detailed description of the contents can be found at the beginning of every section.
\subsection*{Assumptions and notation}
We fix once and for all an algebraically closed field $k_0$, and every scheme is assumed to be of finite type over $\Spec(k_0)$. Every algebraic group is assumed to be linear.

If $X$ is a variety, with the notation $\Het(X)$ we will always mean the graded-commutative ring $\oplus_iH^i_{\'{e}t}(\xi_X,\mu_p^{\otimes i})$,  where $\xi_X$ is the generic point of $X$ and $p$ is a positive number that will be specified in every section.

Sometimes, we will write $\Het(R)$, where $R$ is a finitely generated $k_0$-algebra, to indicate $\Het(\Spec(R))$. Observe that $\Het(k_0)\simeq\FF_p$. 

The Chow groups with coefficients in $\Het$ will be denoted $A^i(-,\Het)$ or $A_i(-,\Het)$, whether we adopt the grading by codimension or dimension. At a certain point we will use the shorthand $A^i(-)$ to denote Chow groups with coefficients in $\Het$ of codimension $i$, and we will drop the index to indicate the direct sum of Chow groups with coefficients of every codimension. A Chow group with coefficients is said to be trivial when it is isomorphic to $\FF_p$.

The $G$-equivariant Chow groups with coefficients in $\Het$ will be denoted $A_G^i(-,\Het)$ or $A_G^i(-)$, and we will write $A^i_G$ to indicate $A^i_G(\Spec(k_0))$. Similar notations will be used for Chow groups $CH^i(-)$. We will write $CH^i(-)_{\FF_p}$ for the tensor product $CH^i(-)\otimes \FF_p$.

We will denote $\bA(n,d)$ the affine space of forms in $n+1$ variables of degree $d$, and $\PP(n,d)$ its projectivization.

Throughout the paper, a relevant role will be played by $\bA(2,2)$, the space of quadratic ternary forms. The subschemes of $\bA(2,2)$ parametrising forms of rank $r$ will be denoted $\bA(2,2)_r$, and the subschemes parametrising forms of rank in the interval $[b,a]$ will be denoted $\bA(2,2)_{[a,b]}$. 

Given a relative scheme $X\to\bA(2,2)$, its pullback to $\bA(2,2)_r$ or $\bA(2,2)_{[a,b]}$ will be denoted $X_r$ or $X_{[a,b]}$. At any rate, these definitions will be frequently repeated along the paper.

Similarly, we will denote $\PP(2,2)_r$ (resp. $\PP(2,2)_{[a,b]}$) the subscheme of $\PP(2,2)$ parametrising conics of rank $r$ (resp. of rank $r$ such that $a\geq r\geq b$). If $X$ is a scheme over $\PP(2,2)$, its pullback to $\PP(2,2)_r$ (resp. $\PP(2,2)_{[a,b]}$) will be denoted $X_r$ (resp. $X_{[a,b]}$).
\subsection*{Acknowledgements}
I wish to thank my advisor Angelo Vistoli for his constant support during this work and for introducing me to this subject. I am also indebted with Roberto Pirisi, for his patience in answering my questions: I think that the relevance of his ideas can be easily detected all along the paper. Finally, I wish to thank the anonymous referees for their careful reading and precious suggestions, which much improved the quality of the exposition.
\section{Equivariant Chow groups with coefficients}\label{sec:equi Rost Chow}
In this section, we fix a positive number $p$ and an algebraically closed base field $k_0$ whose characteristic does not divide $p$. Every scheme is assumed to be of finite type over $k_0$ and equidimensional. 
\subsection{Main definitions and properties}
We will collect together some basic definitions and useful properties of equivariant Chow groups with coefficients in $\Het$. Our interest in these groups is due to the following result:
\begin{thm} \cite[th. 4.9]{PirAlgStack}\label{thm: inv quotients equal to equivariant A0}
	If $X$ is a smooth, quasi-projective scheme endowed with a linearized action of an algebraic group $G$, then we have $$ A^0_G(X,\Het)\simeq \Inv([X/G]) $$
\end{thm}
First, let us sketch the construction of the standard Chow groups with coefficients in $\Het$. The paper of Rost \cite{Rost} is devoted to the foundation of the more general theory of Chow groups with coefficients in a cycle module $M$. Here we are only interested in the case $M=\Het$, though most of what we say is true for any cycle module. The proof that $\Het$ is a cycle module in the sense of \cite{Rost}*{def. 2.1} can be found in \cite{Rost}*{remarks 1.11 and 2.5}.

Another nice introduction to Chow groups with coefficients in $\Het$ is \cite{Guil}*{sec. 2}: in particular, the equivariant case is discussed in \cite{Guil}*{sec. 2.3}. 
\begin{rmk}\label{rmk:assumptions}
	The hypothesis that the characteristic of the base field does not divide $p$ is necessary for $\Het$ to be a cycle module. 
	
	Without this hypothesis, many things can go wrong: for instance, the proof that $\Het$ satisfies the axiom $D3$ (see \cite{Rost}*{remark 1.11}) does not work, as we do not have a norm residue homomorphism.
	
	This is a consequence of the fact that the sequence of sheaves in the \'{e}tale topology:
	\[ 0\longrightarrow \mu_p \longrightarrow \Gm \xrightarrow{\cdot p} \Gm \longrightarrow 0\]
	is not exact if we remove the assumption on the characteristic of the base field (see \cite[\href{https://stacks.math.columbia.edu/tag/03PM}{Tag 03PM}]{stacks-project}).
	
	A possible way to circumvent this issue could be to consider cohomology in the fppf or syntomic topology, but we have not checked all the details.
\end{rmk}
Let $X$ be a scheme, $d$ a positive integer and define:
$$ C_i(X,{\rm H}^d)=\oplus_{x\in X^{(i)}} H^d_{\'{e}t}(k(x),\mu_p^{\otimes d}))$$
where the sum is taken over all the points $x$ of $X$ having dimension equal to $i$. We define $C_i(X,\Het):=\oplus_{d\geq 0} C_i(X,{\rm H}^d)$. In this way $C_i(X,\Het)$ has a natural bigrading given by the cohomological degree $d$ and the dimension $i$. 

For every $i$ ranging from $0$ to the dimension of $X$ and for every $d\geq 1$, there exists a differential:
$$ \delta^d_i:C_i(X,{\rm H}^d)\longrightarrow C_{i+1}(X,{\rm H}^{d-1}) $$
whose precise definition can be found in \cite{Rost}*{(3.2)}. We set $\delta_i^0:=0$. By taking the direct sum over $d\geq 0$ of the $\delta_d^i$, we can define:
$$ \delta_i:C_i(X,\Het)\longrightarrow C_{i+1}(X,\Het) $$
The Chow groups with coefficients are then defined as:
\[A_i(X,\Het):=\ker(\delta_i)/{\rm im}(\delta_{i-1})\]

As we are assuming every scheme to be equidimensional, it makes sense to also introduce the codimensionally-graded Chow groups with coefficients $A^i(X,\Het)$: these are equal by definition to $A_{n-i}(X,\Het)$, where $n$ is the dimension of $X$.

Chow groups with coefficients have two natural gradings, one given by codimension and the other given by the cohomological degree: an element $\alpha$ has codimension $i$ and degree $d$ if it is in $A^i(X,{\rm H}^d):=\ker(\delta^d_{n-i})/\im(\delta^d_{n-i-1})$. 

The whole theory of Chow groups with coefficients has an equivariant counterpart, first introduced in \cite{Guil}*{sec. 2.3}. Let $G$ be an algebraic group acting on a scheme $X$. Suppose moreover that $X$ is quasi-projective with linearized $G$-action. Using the same ideas of \cite{EG}, one can define the equivariant groups $A^i_G(X,\Het)$ as follows: take a representation $V$ of $G$ such that $G$ acts freely on an open subscheme $U\subset V$ whose complement has codimension greater than $i+1$. By \cite{EG}*{lemma 9} such a representation always exists. Then we define:
$$A^i_G(X,\Het):= A^i((X\times U)/G,\Het)$$
The content of \cite{EG}*{prop. 23} assures us that the quotient $(X\times U)/G$ exists in the category of schemes, and by the double filtration argument used in the proof of \cite{EG}*{prop. 1} we see that the definition above does not depend on the choice of $U$.

We list now some properties of equivariant Chow groups with coefficients that will be frequently used throughout the paper:
\begin{prop}\label{pr:properties}
	Let $X$ and $Y$ be equidimensional, quasi-projective schemes of finite type over $k_0$, endowed with a linearized $G$-action. Then we have:
	\begin{enumerate}
		\item $ CH_i^G(X)\otimes \FF_p=A_i^G(X,{\rm H}^0) $.
		
		\item \emph{Proper pushforward}: every $G$-equivariant, proper morphism $f:X\to Y$ induces a homomorphism of groups
		\[f_*:A_i^G(X,\Het)\longrightarrow A_i^G(Y,\Het)\]
		which preserves the cohomological degree.
		
		\item \emph{Flat pullback}:  every $G$-equivariant, flat morphism $f:X\to Y$ of relative constant dimension induces a homomorphism of groups
		\[f^*:A^i_G(Y,\Het)\longrightarrow A^i_G(X,\Het)\]
		which preserves the cohomological degree.
		
		\item \emph{Localization exact sequence}: given a closed, $G$-invariant subscheme $Z\xhookrightarrow{i} X$ whose open complement is $U\xhookrightarrow{j} X$, there exists a long exact sequence
		$$\cdots\to A_i^G(X,\Het) \xrightarrow{j^*} A_i^G(U,\Het) \xrightarrow{\del} A_{i-1}^G(Z,\Het) \xrightarrow{i_*} A_{i-1}^G(X,\Het) \to \cdots $$
		The boundary homomorphism $\del$ has cohomological degree $-1$, whereas the other homomorphisms have cohomological degree zero.
		
		\item \emph{Compatibility}: given a cartesian square of $G$-schemes
		$$ \xymatrix{
			Y \ar[r]^i \ar[d] & X \ar[d] \\
			Y' \ar[r]^{i'} & X' }$$
		where all the morphisms are $G$-equivariant closed embeddings, we get a commutative square
		$$ \xymatrix{
			A_k^G(Y'\setminus Y,\Het) \ar[r]^{i''_*} \ar[d]^{\del} & A_k^G(X'\setminus X,\Het) \ar[d]^{\del} \\
			A_{k-1}^G(Y,\Het) \ar[r]^{i_*} & A_{k-1}^G(X,\Het) } $$
		where $i''$ is the restriction of $i'$ to $Y'\setminus Y$.
		
		\item \emph{Homotopy invariance}: if $\pi:E\to X$ is a $G$-equivariant, finite rank vector bundle, then we have an isomorphism
		$$\pi^*: A^i_G(X,\Het)\simeq A^i_G(E,\Het) $$
		which preserves the cohomological degrees.
		
		\item \emph{Projective bundle formula}: If $\PP(E)\to X$ is the projectivization of a $G$-equivariant , finite rank vector bundle, then for $i< {\rm rk}(E)$ we have: 
		$$ A^i_G(\PP(E),\Het)\simeq \oplus_{j=0}^{i} A^{j}_G(X,\Het)$$
		The isomorphism above preserves the cohomological degrees.
		
		\item \emph{Ring structure}: if $X$ is smooth, then $A_G(X,\Het)$ inherits the structure of a graded-commutative $\FF_p$-algebra, where the graded-commutativity should be understood in the following sense: if $\alpha$ has codimension $i$ and degree $d$, and $\beta$ has codimension $j$ and degree $e$, then $\alpha\cdot\beta=(-1)^{de}\beta\cdot\alpha$, and the product has codimension $i+j$ and degree $d+e$.
	\end{enumerate}
\end{prop}

\begin{proof}[Sketch of proof]
	All the properties above, which are formulated for equivariant Chow groups with coefficients, follow from the analogous properties for the non-equivariant ones, hence it is enough to prove (1)-(7) in this second setting.
	
	To prove (1), observe that $\ker(\delta_i^0)=C_i(X,{\rm H}^0)$. We have $$H^0_{\'{e}t}(k(x),\FF_p)=\FF_p$$ hence $C_i(X,{\rm H}^0)=Z_i(X)\otimes\FF_p$, the group of $i$-dimensional cycles mod $p$. 
	
	Recall that $H^1_{\'{e}t}(k(y),\mu_p)=k(y)^*/(k(y)^*)^p$, so that the group $C_{i-1}(X,{\rm H}^1)$ can be seen as a quotient of the group of non-zero rational functions on $i-1$-dimensional subvarieties of $X$. 
	
	Unwinding the definition of the differential $\delta_{i-1}^1$ (see \cite{Rost}*{(3.2)}), we deduce that $\delta_{i-1}^1(\overline{\varphi})={\rm div}(\overline{\varphi})$, where $\overline{\varphi}$ is the equivalence class in $k(y)^*/(k(y)^*)^p$ of a rational function $\varphi$ defined over a subvariety $Y\subset X$ whose generic point is $y$ , and ${\rm div}(\overline{\varphi})$ is the associated divisor, which is a well defined element of $Z_i(X)\otimes\FF_p$.
	This implies that:
	\[ (Z_i(X)/\sim_{\rm rat})\otimes\FF_p \simeq \ker(\delta_i^0)/\im(\delta_{i-1}^1) \]
	where $\sim_{\rm rat}$ is the rational equivalence relation, and proves (1).
	
	The proofs of (2)-(5) are the content of \cite{Rost}*{sec. 4}, (6) is \cite{Rost}*{prop. 8.6}, (7) is \cite{PirCohHypEven}*{prop. 2.4} and (8) is \cite{Rost}*{th. 14.6}.
\end{proof}
There is also a well defined theory of equivariant Chern classes for equivariant Chow groups with coefficients in $\Het$, which resembles very much the theory of Chern classes for the usual Chow groups. The foundation of this theory can be found in \cite{PirCohHypEven}*{sec. 2.1},

Let $X$ be a quasi-projective scheme endowed with a linearized $G$-action and let $E\to X$ be an equivariant vector bundle of rank $r$. Then for every $i$ the vector bundle $E$ induces a degree-preserving homomorphism of equivariant Chow groups with coefficients:
\[ c_i^G(E):A_*^G(X,\Het) \longrightarrow A_{*-i}^G(X,\Het)\]
which is the $i^{\rm th}$ \emph{equivariant Chern class} of $E$.

Using the language of Chern classes, the projective bundle formula \ref{pr:properties}.(7) can be restated (see \cite{PirCohHypEven}*{prop. 2.4}) by saying that, if $\pi:\PP(E)\to X$ is the projectivization of an equivariant vector bundle, then we have:
\[ A^i_G(\PP(E),\Het)\simeq \oplus_{j=0}^{i} c_1^G(\OO_{\PP(E)}(1))^{i-j}(\pi^*A_G^{j}(X,\Het)) \]

If $X$ is smooth, there exist cycles $\gamma_i$ in $A_G^i(X,{\rm H}^0)$ such that the equivariant Chern classes are equal to the multiplication by $\gamma_i$ (see \cite{PirCohHypEven}*{cor. 2.6}). Therefore, in this case it makes perfect sense to think of equivariant Chern classes as elements in $A_G(X,\Het)$.

In particular, for $\pi:\PP(E)\to X$ the projectivization of an equivariant vector bundle over a smooth scheme, we can consider the element $h=c_1^G(\OO(1))$ in $A^1_G(\PP(E),{\rm H}^0)$ and proposition \ref{pr:properties}.(7) can be reformulated as follows:
$$ A_G(\PP(E),\Het)\simeq A_G(X,\Het)[h]/(f) $$
where $f$ is an element of $A(X,\Het)[h]$ monic of codimensional degree equal to the rank of $E$. The isomorphism preserves the cohomological degree.
\subsection{Examples}
We collect here some computations of equivariant Chow groups with coefficients which will be useful for our purposes.
\begin{prop}\cite[prop. 2.11]{PirCohHypEven}\cite[prop. 3.1, prop. 3.2]{PirCohHypThree}\label{pr:Rost of BPGL_2 and P1}
	We have:
	\begin{enumerate}
		\item Let $E$ be the standard representation of $\GLt$, regarded as a $\GLt$-equivariant vector bundle of rank $3$ over $\Spec(k_0)$. Then $$\Agl(\Spec(k_0),\Het)=\FF_p[c_1,c_2,c_3]$$
		where $c_i:=c_i^{\GLt}(E)$. In particular, $\Agl(\Spec(k_0),\Het)$ is concentrated in degree $0$.
		\item Fix $p=2$ and let $E$ be the standard representation of $\GLt$, on which $\PGLt$ acts via the adjoint representation. 
		Set $c_i:=c_i^{\PGLt}(E)$. Then $$\Apgl(\Spec(k_0),\Het)=\FF_2[c_2,c_3] \oplus\FF_2[c_2,c_3]\cdot w_2 \oplus \FF_2[c_2,c_3]\cdot \tau_1 $$ 
		as an $\FF_2$-vector space. The generator $w_2$ has codimension $0$ and degree $2$, the generator $\tau_1$ has codimension $1$ and degree $1$ and $c_1=0$.
		\item Fix $p=2$. Then $\Apgl^0(\PP^1,\Het)\simeq \FF_2$.
	\end{enumerate}
\end{prop}
\begin{rmk}
	Proposition \ref{pr:properties}.(7) and \ref{pr:Rost of BPGL_2 and P1}.(3) are not in contradiction, as $\PP^1$ is not the projectivization of any $\PGLt$-representation of rank 2.
\end{rmk}
Take a quasi-projective scheme $X$ endowed with a linearized $G$-action and let $L\to X$ be an equivariant line bundle. Define $L^*$ as the complement in $L$ of the zero section. Observe that $L^*$ inherits the $G$-action and that it has a natural structure of $\Gm$-torsor, where $\Gm$ acts by scalar multiplication. 

This second fact can also be seen by noting that $L^*\simeq \underline{\rm Isom}_X(L,\bA^1_X)$ and that $\Gm$ acts on the second scheme as follows: given a morphism $S\to X$, a trivialization $\varphi:L_S\simeq \bA^1_S$ and an element $\lambda:\bA^1_S\simeq\bA^1_S$ of $\Gm(S)$, we can define $\lambda\cdot\varphi:=\lambda\circ\varphi$.

The next result tells us that the equivariant Chow groups with coefficients of $L^*$ depend only on the equivariant Chow groups with coefficients of $X$ and on the morphism $c_1^G(L)$.
\begin{prop}\label{pr:Gm torsor formula}
	Let $X$ be an equidimensonal, quasi-projective scheme over $k_0$, endowed with a linearized $G$-action, and let $L\to X$ be an equivariant line bundle, with associated $\Gm$-torsor $L^*$. Then we have:
	\[A_G(L^*,\Het)\simeq \left(A_G(X,\Het)/\im(c_1^G(L))\right)\oplus \ker(c_1^G(L))[1]\]
	as $\FF_p$-vector spaces.
\end{prop}
\begin{proof}
	Here $A_G^n(-)$ stands for $A_G^n(-,\Het)$.
	Using proposition \ref{pr:properties}.(4) applied to the zero section $i:X\hookrightarrow L$, we get:
	\[\cdots\to A_G^{n-1}(X)\xrightarrow{i_*} A_G^n(L) \to A_G^n(L^*) \to A_G^n(X) \xrightarrow{i_*} A_G^{n+1}(X)\to\cdots\]
	From proposition \ref{pr:properties}.(6), we know that there exists an isomorphism $(\pi^*)^{-1}:A_G(L)\simeq A_G(X)$. By definition of the Chern classes (see in particular \cite{PirCohHypEven}*{def. 2.2}) we have $c_1^G(L)=(\pi^*)^{-1}\circ i_*$, so that the exact sequence above can be rewritten as:
	\[\cdots\to A_G^{n-1}(X)\xrightarrow{c_1^G(L)} A_G^n(X) \to A_G^n(L^*) \to A_G^n(X) \xrightarrow{c_1^G(L)} A_G^{n+1}(X)\to\cdots\]
	This implies the formula for $A_G(L^*)$.
\end{proof}
\section{Cohomological invariants of $\Hcal_g$}\label{sec:coh inv}
In this section, we fix $p=2$, so that $\Het(-)=\oplus_{d\geq 0} H_{\'{e}t}^d(-,\mu_2^{\otimes d})$. We keep assuming the base field $k_0$ to be algebraically closed and of characteristic $\neq 2$. We will also adopt the shorter notation $A_G(-)$ instead of $A_G(-,\Het)$ to denote Chow groups with coefficients.

A \textit{family of hyperelliptic curves of genus }$g$ is a pair $(C\to S,\iota)$ where $C\to S$ is a proper and smooth morphism whose fibres are curves of genus $g$, and $\iota\in{\rm Aut}(C)$ is an involution such that the quotient $C/\langle\iota\rangle\to S$ is a proper and smooth morphism whose fibres are curves of genus $0$.

Therefore, the stack of hyperelliptic curves of genus $g$ is defined as the stack in groupoids $\Hcal_g$ over the site ${\rm Sch}/k_0$ whose objects are families of hyperelliptic curves.
In this section we will prove our main theorem, which is the following:
\begin{thm}\label{thm:coh inv of Hg}
	Let $k_0$ be an algebraically closed field of characteristic $\neq 2$, and let $\Hcal_g$ denote the moduli stack of smooth hyperelliptic curves of odd genus $g\geq 3$ over $k_0$.
	
	Then the graded-commutative ring of cohomological invariants $\Inv(\Hcal_g)$ with coefficients in $\FF_2$ is generated, as a graded $\FF_2$-vector space, by classes $$1,x_1,w_2,x_2,...,x_{g+1},x_{g+2}$$ where the degree of each $x_i$ is $i$ and $w_2$ is the second  Stiefel-Whitney class coming from $\Inv(\cl{B}\PGLt)$.
\end{thm}
The case $g=3$ is \cite{PirCohHypThree}*{th. 0.1}. Actually, the only obstruction to generalize the result contained there to any odd genus is given by \cite[corollary 3.9 ]{PirCohHypThree}, where the assumption $g=3$ is strictly necessary. Once one generalizes that corollary, the computation of the cohomological invariants is basically done. 

Therefore, what we present here is essentially a rewriting of the proof contained in \cite{PirCohHypThree}*{sec. 3}: the only difference is in the key lemma \ref{lm:key}, which is a generalization of \cite[corollary 3.9 ]{PirCohHypThree}. The proof of this key result, which is rather non-trivial, is postponed to section \ref{sec:main lemma}, as we need to develop more theory in order to complete it.

\subsection{Setup} Let $\bA(1,n)$ be the affine space of binary forms of degree $n$ and let $X_n$ be the open subscheme parametrising forms with distinct roots. In \cite{ArsVis}*{cor. 4.7} the authors gave the following presentation of $\Hcal_g$ as a quotient stack, when $g\geq 3$ is an odd number:
$$\Hcal_g \simeq [X_{2g+2}/(\PGLt\times\Gm)] $$
The action of $\PGLt\times\Gm$ on $X_{2g+2}$ descends from the action of ${\rm GL}_2\times\Gm$ defined by the formula:
$$(A,\lambda)\cdot f(x,y):=\lambda^{-2}\det(A)^{g+1}f(A^{-1}(x,y))$$
This presentation and theorem \ref{thm: inv quotients equal to equivariant A0} imply that:
\begin{equation}\label{eq:inv}
	\Inv(\Hcal_g)\simeq A^0_{\PGLt\times\Gm}(X_{2g+2})
\end{equation}
Therefore, to obtain the cohomological invariants of $\Hcal_g$, it is enough to compute the codimension $0$ part of an equivariant Chow ring with coefficients.

Let $\PP(1,2n)$ be the projective space of binary forms of degree $2n$, and denote $\Delta_{1,2n}$ the divisor in $\PP(1,2n)$ parametrising singular forms. We are first going to compute $A^0_{\PGLt}(\PP(1,2g+2)\setminus\Delta_{1,2g+2})$, and then we will use the fact that
$$ X_{2g+2}\longrightarrow \PP(1,2g+2)\setminus \Delta_{1,2g+2} $$
is a $\PGLt\times\Gm$-equivariant $\Gm$-torsor to deduce a presentation of $A^0_{\PGLt\times\Gm}(X_{2g+2})$.

\subsection{Proof of the main theorem} 
As already stated at the beginning of this section, we will always be assuming $p=2$. The proposition below is the starting point to determine $\Inv(\Hcal_g)$.
\begin{prop}\label{pr: inv of P2n minus Delta}\cite{PirCohHypThree}*{cor. 3.10}
	The graded-ring $\Apgl^0(\PP(1,2n)\setminus\Delta_{1,2n})$ is freely generated as $\FF_2$-module by $n+1$ elements $x_1,...,x_n,w_2$ where the degree of $x_i$ is $i$ and $w_2$ is the second Stiefel-Whitney class coming from the cohomological invariants of $\cl{B}\PGLt$.
\end{prop}
The proof of this proposition is by induction on $n$. To set up the induction argument, we need the following technical lemma, which is of fundamental importance:
\begin{kl}\label{lm:key}
	Let $n\geq 1$ be an integer and let $k_0$ be an algebraically closed field of characteristic $\neq 2$. Let $i:\Delta_{1,2n}\hookrightarrow\PP(1,2n)$ be the inclusion of the subscheme of singular forms into the projective space of binary forms of degree $2n$ over $k_0$. Then the pushforward homomorphism:
	\[i_*: \Apgl^0(\Delta_{1,2n},\Het)\longrightarrow\Apgl^1(\PP(1,2n),\Het)\]
	between equivariant Chow groups with coefficients in $\Het:=\oplus_i H^i_{\'{e}t}(-,\mu_2^{\otimes i})$ vanishes.
\end{kl}
The content of the key lemma above is equivalent to saying that for $n\geq 1$ the boundary morphism $$\del:\Apgl^0(\PP(1,2n))\to\Apgl^0(\Delta_{1,2n})$$ is surjective.

This is proved by Pirisi only for $n\leq 8$ (see \cite{PirCohHypThree}*{cor. 3.9}): he shows that, for every $n\geq 1$ there exists an element $g_{2n}$ in $\Apgl(\PP(1,2n))$ such that, for every $\alpha$ in $\Apgl^0(\Delta_{1,2n})$, we have $g_{2n}\cdot i_*\alpha=0$ (what here is called $\PP(1,2n)$ is denoted $P^{2n}$ by Pirisi).

From this, he deduces that if $i_*\alpha\neq0$ then $g_{2n}\cdot c_1^G(\OO(1))$ is divided by another element $f_{2n}$ of $\Apgl(\PP(1,2n))$, and this cannot be the case for $2n\leq 8$ but can be true for $2n>8$: the last two assertions follow from the explicit construction of $f_{2n}$ and $g_{2n}$ contained in \cite{PirCohHypThree}*{lm. 3.7, pr. 3.8}.

In this way, Pirisi proves that $i_*=0$ for $2n\leq 8$, which implies the key lemma \ref{lm:key} for $n\leq 4$. This allows him to prove proposition \ref{pr: inv of P2n minus Delta} only for this values of $n$, and to compute $\Inv(\Hcal_g)$ only for $g=3$.

In section \ref{sec:main lemma} we will prove key lemma \ref{lm:key} for every $n\geq 1$, adopting a completely different strategy. For the remainder of the section, we assume key lemma \ref{lm:key}. 

We will also need the following results:
\begin{lm}\label{lm:setup induction}
	We have:
	\begin{enumerate}
		\item for $n\geq 2$, there is an isomorphism $$\Apgl^0(\Delta_{1,2n})\simeq\Apgl^0((\PP(1,2n-2)\setminus\Delta_{1,2n-2})\times\PP^1)$$
		\item for $n\geq 1$, the pullback morphism $$\Apgl^0(\PP(1,2n)\setminus\Delta_{1,2n})\to\Apgl^0((\PP(1,2n)\setminus\Delta_{1,2n})\times\PP^1)$$ is surjective with kernel generated by $w_2$, the Stiefel-Whitney class coming from the cohomological invariants of $\PGLt$.
	\end{enumerate}
\end{lm}
\begin{proof}
	Before proceeding with the proof, let us warn the reader that what is called $\PP(1,2n)$ here is denoted $P^{2n}$ in \cite{PirCohHypThree}.
	
	From \cite{PirCohHypThree}*{pr. 2.2} we know that $$\Apgl^0(\Delta_{1,2n}\setminus\Delta_{2,2n})\simeq \Apgl^0((\PP(1,2n-2)\setminus\Delta_{1,2n-2})\times\PP^1)$$ where $\Delta_{2,2n}$ is the closed subscheme of $\PP(1,2n)$ parametrising forms which have two double roots. To prove (1), it is then enough to show that $\Apgl^0(\Delta_{1,2n})\simeq \Apgl^0(\Delta_{1,2n}\setminus\Delta_{2,2n})$.
	
	The key lemma \ref{lm:key} together with \cite{PirCohHypThree}*{cor. 3.5} imply that the condition $S_1(2n)$ (resp. $S_2(2n)$) of \cite{PirCohHypThree}*{pg. 14} hold true for every $n\geq 1$ (resp. for every $n\geq 2$). 
	
	Condition $S_1(2n)$ is exactly (2), and condition $S_2(2n)$ says that the pullback morphism $\Apgl^0(\Delta_{1,2n})\to\Apgl^0(\Delta_{1,2n}\setminus\Delta_{2,2n})$ is an isomorphism, from which we deduce (1). 
\end{proof}

\begin{proof}[Proof of prop. \ref{pr: inv of P2n minus Delta}]	
	The key lemma \ref{lm:key}, applied to the localization exact sequence (see proposition \ref{pr:properties}.(4)) associated to the closed subscheme $\Delta_{1,2n}$, gives us the following short exact sequence of $\FF_2$-vector spaces:
	$$ 0\to \Apgl^0(\PP(1,2n))\to \Apgl^0(\PP(1,2n)\setminus\Delta_{1,2n})\xrightarrow{\del}\Apgl^0(\Delta_{1,2n})\to 0$$
	From this we deduce:
	$$ \Apgl^0(\PP(1,2n)\setminus\Delta_{1,2n})\simeq \Apgl^0(\PP(1,2n))\oplus\Apgl^0(\Delta_{1,2n})[1] $$
	where the notation $\Apgl(\Delta_{1,2n})[1]$ means that everything is degree-shifted by one.
	
	The projective space $\PP(1,2n)$ is the projectivization of a $\PGLt$-representation, hence using the projective bundle formula (proposition \ref{pr:properties}.(7)) we deduce $$\Apgl^0(\PP(1,2n))\simeq\Apgl^0$$ and the $\FF_2$-vector space on the right is known (see proposition \ref{pr:Rost of BPGL_2 and P1}.(2)). 
	
	To compute $\Apgl^0(\Delta_{1,2n})$, we proceed by induction on $n$. To prove the base case $n=1$, observe that $\Delta_{1,2}\simeq\PP^1$, and $\Apgl^0(\PP^1)$ is trivial by proposition \ref{pr:Rost of BPGL_2 and P1}.(3).
	
	Lemma \ref{lm:setup induction}.(1) tells us that:
	\[\Apgl^0(\Delta_{1,2n})\simeq \Apgl^0((\PP(1,2n-2)\setminus\Delta_{1,2n-2})\times\PP^1) \]
	and lemma \ref{lm:setup induction}.(2) says that the $\FF_2$-vector space on the right is isomorphic to $\Apgl^0(\PP(1,2n-2)\setminus\Delta_{1,2n-2})/\FF_2\cdot w_2$. Using the inductive hypothesis, we get the desired result.
\end{proof}

As announced at the beginning of the section, we will use proposition \ref{pr: inv of P2n minus Delta} to compute $A^0_{\PGLt\times\Gm}(X_{2n})$. We will use the following result:
\begin{lm}\label{lm:from PGL2 to PGLxGm} \cite[proposition 2.3]{PirCohHypThree}
	Let $Y$ be a scheme endowed with an action of $\PGLt$, and let $\Gm$ act trivially on it. Then
	$$A_{\PGLt\times\Gm}(Y)\simeq \Apgl(Y)[t]$$
	where $t$ has codimension $1$ and degree $0$. In particular, we get an isomorphism on the codimension zero part.
\end{lm}
We observed that $X_{2n}\to\PP(1,2n)\setminus\Delta_{1,2n} $
is a $\PGLt\times\Gm$-equivariant $\Gm$-torsor. Let $L$ be the line bundle associated to $X_{2n}$.
\begin{lm}\label{lm:new inv coming from Gm-torsor} 
	The codimension zero part of $\ker(c_1^{\PGLt\times\Gm}(L))$ is generated as an $\FF_2$-vector space by a single element $x_n$ of degree $n$.
\end{lm}
\begin{proof}
	The proof of \cite{PirCohHypThree}*{th. 3.12}, once we know the key lemma \ref{lm:key} for every $n\geq 1$, works for all $n\geq 1$.
\end{proof}
We now have all the elements necessary to prove the main result of the paper.
\begin{proof}[Proof of theorem \ref{thm:coh inv of Hg}]
	We know that $\Inv(\Hcal_g)\simeq A^0_{\PGLt\times\Gm}(X_{2g+2})$ (see (\ref{eq:inv})). Applying proposition \ref{pr:Gm torsor formula} to the $\Gm$-torsor $X_{2g+2}\to\PP(1,2g+2)\setminus\Delta_{1,2g+2}$, we get:
	\[ A_{\PGLt\times\Gm}^0(X_{2g+2})\simeq A^0_{\PGLt\times\Gm}(\PP(1,2g+2)\setminus\Delta_{1,2g+2})\oplus \ker(c_1^{\PGLt\times\Gm}(L))[1] \]
	where $L$ is the line bundle on $\PP(1,2g+2)\setminus\Delta_{1,2g+2}$ associated to the $\Gm$-torsor $X_{2g+2}$, and we consider only the codimension zero part of $\ker(c_1^{\PGLt\times\Gm}(L))$.
	
	The first summand is isomorphic to $\Apgl^0(\PP(1,2g+2)\setminus\Delta_{1,2g+2})$ by lemma \ref{lm:from PGL2 to PGLxGm}, and this $\FF_2$-vector space can be computed using proposition \ref{pr: inv of P2n minus Delta} with $n=g+1$.
	
	The second summand is also known, thanks to lemma \ref{lm:new inv coming from Gm-torsor}. Putting all together, we get the desired conclusion.
\end{proof}
\section{$\GLt$-counterpart of $\PGLt$-schemes}\label{sec:GL3 counterpart}
In this section we introduce the notion of $\GLt$-counterpart of a $\PGLt$-scheme (definition \ref{def:GLt counterpart of a morphism}). We also state a simple result (proposition \ref{prop: chow diagrams}) which enables us to replace morphisms between $\PGLt$-equivariant Chow groups with coefficients of $\PGLt$-schemes with morphisms between $\GLt$-equivariant Chow groups with coefficients of their $\GLt$-counterparts. 

This fact will be quite relevant in the proof of the key lemma \ref{lm:key}.

As already said in the introduction, we will adopt the following notation: for $n\geq 1$ an integer, we will denote $\bA(n,d)$ the affine space of homogeneous forms in $n+1$ variables of degree $d$, and $\PP(n,d)$ will stand for the projectivization of $\PP(n,d)$.

We will denote $\bA(2,2)_r$ the locally closed subscheme of $\bA(2,2)$ of forms of rank $r$, and $\bA(2,2)_{[a,b]}$ will stand for the subscheme of forms of rank $r$ with $a\geq r\geq b$. For instance, the scheme $\bA(2,2)_3$ parametrizes smooth forms.

Given a scheme $X$ over $\bA(2,2)_{[3,1]}$, we will denote the pullback of $X$ to $\bA(2,2)_r$ (resp. $\bA(2,2)_{[a,b]}$) as $X_r$ (resp. $X_{[a,b]}$). 
\subsection{Basic definitions and some properties}\label{subsec:GL3 counterpart}
Let $f:X\to X'$ be a $\PGLt$-equivariant morphism between two schemes of finite type over a base field $k_0$. We can form the quotient stacks $[X/\PGLt]$ and $[X'/\PGLt]$.

The morphism $f$ induces a representable morphism between the two quotient stacks. Observe that both $[X/\PGLt]$ and $[X'/\PGLt]$ are stacks over $\Bcal\PGLt$, and that both structure morphisms are representable.

We have:
\begin{prop}\cite{Dil}*{pr. 1.5}
	Let $\bA(2,2)_3$ denote the scheme parametrising smooth quadratic ternary forms, endowed with the $\GLt$-action defined by the formula:
	\[ A\cdot q(x,y,z):=\det(A)q(A^{-1}(x,y,z)) \]
	Then $[\bA(2,2)_3/\GLt]\simeq \Bcal\PGLt$.
\end{prop}
In other terms, there is a morphism $\bA(2,2)_3\to\Bcal\PGLt$ which is a $\GLt$-torsor. We can pull back along this torsor the map $[X/\PGLt]\to [X'/\PGLt]$: what we get is a $\GLt$-equivariant morphism $g:Y\to Y'$ between $\GLt$-schemes.
\begin{df}\label{def:GLt counterpart of a morphism}\cite{Dil}*{def. 1.1 and 1.2}
	Let $X$ and $X'$ be two schemes of finite type over ${\rm Spec}(k_0)$ endowed with a $\PGLt$-action, and let $f:X\to X'$ be a $\PGLt$-equivariant morphism. Then the $\GLt$-equivariant morphism $g:Y\to Y'$ obtained with the construction above is the $\GLt$\emph{-counterpart of} $f$.
	
	The $\GLt$-scheme $Y$ (resp. $Y'$) is the $\GLt$\emph{-counterpart of} $X$ (resp. $X'$).
\end{df}
The construction of $\GLt$-counterparts is functorial. Observe also that if $f$ is proper (resp. flat) then $g$ will also be proper (resp. flat).

There is another way to think of $\GLt$-counterparts (cf. \cite{Dil}*{rem. 1.6}). The inclusion $i:\PGLt\hookrightarrow\GLt$ induces a representable morphism of classifying stacks $\Bcal\PGLt\to\Bcal\GLt$: this morphism sends a $\PGLt$-torsor $P\to S$ to the associated $\GLt$-torsor $P\times^{\PGLt}\GLt\to S$, where $P\times^{\PGLt}\GLt:=(P\times \GLt)/\PGLt$ and the (right) action on the product is defined by the formula:
\[ (x,g)\cdot h:=(h^{-1}x,gi(h)) \]
In general, given an algebraic group $G$, the functor $X\mapsto [X/G]$ induces an equivalence between the category $G$-equivariant schemes over $k_0$ and the category of algebraic stacks over $\Bcal G$ with representable structure morphism: therefore, given a $\PGLt$-scheme $X$, we get a representable morphism $[X/\PGLt]\to\Bcal\PGLt\to\Bcal\GLt$, hence a $\GLt$-equivariant scheme over $k_0$.

The following proposition is immediate to prove:
\begin{prop} \label{prop: chow diagrams}
	Let $f:X\to X'$ be a $\PGLt$-equivariant proper morphism between two $\PGLt$-schemes, and let $g:Y\to Y'$ be its $\GLt$-counterpart. Then we have:
	\begin{enumerate}
		\item a commutative diagram of equivariant Chow groups of the form
		$$\xymatrix {
			CH^{\PGLt}_i(X) \ar[r]^{f_*} \ar[d] & CH^{\PGLt}_i(X') \ar[d] \\
			CH^{\GLt}_i(Y) \ar[r]^{g_*} & CH^{\GLt}_i(Y') }$$
		where the vertical arrows are isomorphisms.\\
		\item a commutative diagram of equivariant Chow groups with coefficients of the form
		$$\xymatrix {
			A^{\PGLt}_i(X) \ar[r]^{f_*} \ar[d] & A^{\PGLt}_i(X') \ar[d] \\
			A^{\GLt}_i(Y) \ar[r]^{g_*} & A^{\GLt}_i(Y') }$$
		where the vertical arrows are isomorphisms.
	\end{enumerate} 	
\end{prop}
\subsection{Applications}
We now apply the machinery above to a particular case. Let $\PP(1,2n)$ be the projective space of binary forms of degree $2n$. This scheme has a natural action of $\PGLt$ given by $A\cdot f(x,y)=f(A^{-1}(x,y))$. We want to find its $\GLt$-counterpart.

Consider the scheme $\bA(2,2)_{[3,1]}=\bA(2,2)\setminus\{0\}$ of non-zero quadratic ternary forms. This is a $\Gm$-torsor over the projective space $\PP(2,2)$ of plane conics.

\begin{df}\label{df:universal conic}
	\hspace{1pt}
	\begin{enumerate}
		\item We denote $Q\subset\PP(2,2)\times\PP^2$ the universal conic over $\PP^2$.
		\item We denote $\widehat{Q}\subset \bA(2,2)_{[3,1]}\times\PP^2$ the pullback of the universal conic $Q\to\PP(2,2)$ along the $\Gm$-torsor $\bA(2,2)_{[3,1]}\to\PP(2,2)$.
	\end{enumerate}
\end{df}
Let $\pr_1$ and $\pr_2$ be respectively the projection on the first and on the second factor of $\PP(2,2)\times\PP^2$. We have a short exact sequence of $\OO_{\PP(2,2)\times \PP^2}$-modules:
\[0\longrightarrow \Ical_Q\otimes\pr_2^*\OO(n) \longrightarrow \pr_2^*\OO(n)\longrightarrow \pr_2^*\OO(n)|_Q \longrightarrow 0\]
where $\Ical_Q\simeq\pr_1^*\OO(-1)\otimes\pr_2^*\OO(-2)$ is the ideal sheaf of $Q$.

Pushing everything forward along $\pr_1$, and applying the projection formula for sheaves (see \cite{Har}*{ex. III.8.3}) we get:
\[ 0\to \pr_{1*}\pr_2^*\OO(n-2)\otimes\OO(-1)\to \pr_{1*}\pr_2^*\OO(n) \to \pr_{1*}(\pr_2^*\OO(n)|_Q) \to 0 \]
where exactness on the right is due to the vanishing of $R^1\pr_{1*}\pr_2^*\OO(n-2)$, which follows from the base change theorem in cohomology (see \cite{Har}*{th. III.12.11}) combined with the well known vanishing of $H^1(\PP^2,\OO(n-2))$.

Observe that for every point $p$ of $\PP(2,2)$, the rank of $H^0(Q_p,\OO(n)|_{Q_p})$ is constant: applying again the base change theorem in cohomology we deduce that $\pr_{1*}(\pr_2^*\OO(n)|_Q)$ is locally free.
\begin{df}\label{def:Vn}\cite{Dil}*{def. 2.1}
	\hspace{1pt}
	\begin{enumerate}
		\item We define $\overline{V}_n$ as the vector bundle associated to the locally free sheaf $\pr_{1*}(\pr_2^*\OO(n)|_Q)$. Its projectivization is denoted $\PP(\ol{V}_n)$.
		\item We define $V_n\to\bA(2,2)_{[3,1]}$ as the vector bundle obtained by pulling back $\overline{V}_n$ along the $\Gm$-torsor $\bA(2,2)_{[3,1]}\to\PP(2,2)$. Its projectivization is denoted $\PP(V_n)$.
	\end{enumerate}	
\end{df}
Equivalently, $V_n$ is the vector bundle associated to $\pr_{1*}(\pr_2^*\OO(n)|_{\widehat{Q}})$.

Another way to think of the vector bundle $V_n$ is as follows: consider the following injective morphism of (trivial) vector bundles over $\bA(2,2)_{[3,1]}$:
$$ \bA(2,2)_{[3,1]}\times \bA(2,n-2)\longrightarrow \bA(2,2)_{[3,1]}\times \bA(2,n),\quad (q,f)\longmapsto (q,qf) $$
Then the vector bundle $V_n$ is equal to the cokernel of the morphism above.
\begin{rmk}\label{rmk:points of Vn}
	The points of $\PP(\ol{V}_n)$ can be thought as pairs $(q,[f])$, where $q$ is the projective equivalence class of a non-zero quadratic ternary form (equivalently, a conic) and $[f]$ is the equivalence class of a non-zero ternary form of degree $n$, where $f\sim f'$ if and only if $g$ divides $f-f'$ or $f'$ is a non-zero scalar multiple of $f$. The same description holds for the points of $\PP(V_n)$, with the exception that in this case we really look at the ternary quadratic form $q$ and not at the associated conic.
\end{rmk}

\begin{prop}\label{prop:P(Vn) counterpart of P^2n}
	The $\GLt$-counterpart of $\PP(1,2n)$ is $\PP(V_n)_3$, endowed with the $\GLt$-action:
	$$A\cdot (q,[f]):=(\det(A)q(A^{-1}(x,y,z)),[f(A^{-1}(x,y,z))])$$
	where $q$ is a smooth ternary forms of degree $2$ and $f$ is a representative of the equivalence class $[f]$ of a ternary form of degree $n$.
\end{prop}
\begin{proof}
	See \cite[proposition 2.4]{Dil} (what we call here $\PP(V_n)_3$ is denoted $\PP(V_n)$ there).
\end{proof}
\begin{rmk}\label{rmk:hilb}
	Another way to think of the points of $\PP(V_n)_3$ is as pairs $(q,E)$, where $E$ is an effective divisor of degree $2n$ of the smooth plane conic $C$ defined by the equation $q=0$. 
	
	Let $F$ and $G$ be the plane curves respectively defined by the equations $f=0$ and $g=0$, and suppose that they do not contain $C$ as an irreducible component.
	
	By the classical Noether's theorem $AF+BG$, the intersection of $F$ with $C$ is equal to the intersection of $G$ with $C$ if and only if the difference $f-g$ is divisible by $q$, or in other terms if and only if $f-g$ is in the image of $\bA(2,2)_3\times\bA(2,n-2)\to\bA(2,2)_3\times\bA(2,n)$.
	
	From this we deduce that the points of $\PP(V_n)_3$ are in bijection with the pairs $(q,E)$, where $E$ is an effective divisor of degree $2n$.
	
	This can be reformulated by saying that $\PP(V_n)_3$ is the relative Hilbert scheme of points ${\rm Hilb}^{2n}_{\widehat{Q}_3/\bA(2,2)_3}$, where $\widehat{Q}_3$ is the pullback of the universal conic $Q$ over $\PP(2,2)$.
	
	Similarly, the scheme $\PP(\ol{V}_n)_3$ is the relative Hilbert scheme ${\rm Hilb}^{2n}_{Q_3/\PP(2,2)_3}$
\end{rmk}
We now sketch a construction contained in \cite{Dil}*{pg. 10}.

Let $\rho:Q\to\PP(2,2)$ be the universal conic. From definition \ref{def:Vn}.(1) we have:
\[ \overline{\Vcal}_n=\rho_*(\pr_2^*\OO(n)|_{Q})\]
We can form the following cartesian diagram:
\[
\xymatrix{
	Q'\ar[r]^{\psi} \ar[d]^{\varphi} & Q \ar[d]^{\rho} \\
	\PP(\overline{V}_n) \ar[r]^{\pi} &\PP(2,2)
}
\]
We have the canonical injection:
\[ \OO_{\PP(\overline{V}_n)}(-1) \hookrightarrow \pi^*\overline{\Vcal}_n \]
This can be pulled back along $\varphi$, so that we obtain:
\[ \alpha:\varphi^*\OO_{\PP(\overline{V}_n)}(-1)\hookrightarrow \varphi^*\pi^*\overline{\Vcal}_n=\psi^*\rho^*\overline{\Vcal}_n  \]
Observe that we have a canonical surjection:
\[ \rho^*\overline{\Vcal}_n=\rho^*\rho_*(\pr_2^*\OO(n)|_{Q})\longrightarrow \pr_2^*\OO(n)|_{Q} \]
Pulling back the morphism above along $\psi$, composing it with $\alpha$ and tensoring it with $\pr_2^*\OO(-n)|_{Q'}$, we finally get:
\[ \varphi^*\OO_{\PP(\overline{V}_n)}(-1)\otimes\pr_2^*\OO(-n)|_{Q'}\longrightarrow \OO_{Q'} \]
The image of the morphism above is a sheaf of ideals $\Ical_{\overline{\Dcal}}$. Let $\overline{\Dcal}\subset Q'$ be the subscheme associated to $\Ical_{\overline{\Dcal}}$: then the fibre of $\overline{\Dcal}$ over a point $(\ol{q},[f])$ of $\PP(\overline{V}_n)$ is equal to the subscheme of $\PP^2$ associated to the homogeneous ideal $I=(q,f)$, which is well defined.

Observe that $\overline{\Dcal}\to\PP(\overline{V}_n)$ is generically \'{e}tale of degree $2n$.
\begin{rmk}
	If we think of $\PP(\ol{V}_n)_3$ as the relative Hilbert scheme of points ${\rm Hilb}^{2n}_{Q_3/\PP(2,2)_3}$ (see remark \ref{rmk:hilb}), then the restriction $\ol{\Dcal}_3$ of $\ol{\Dcal}$ over $\PP(\ol{V}_n)_3$ is exactly the universal subscheme associated to the Hilbert scheme.
\end{rmk}
\begin{df}\label{def:D}
	\hspace{1pt}
	\begin{enumerate}
		\item Let $\overline{\Dcal}^{\rm ram}$ be the locus defined by the $0^{\rm th}$ Fitting ideal of the sheaf of relative differentials $\Omega^1_{\overline{\Dcal}/\PP(\overline{V}_n)}$. Then we define $\overline{D}$ as the image of the projection of $\overline{\Dcal}^{\rm ram}$ onto $\PP(\overline{V}_n)$.
		\item We define \emph{the fundamental divisor} $D$ as the pullback of $\overline{D}\to\PP(2,2)$ along the morphism $\bA(2,2)_{[3,1]}\to\PP(2,2)$.
	\end{enumerate}
\end{df}
\begin{rmk}\label{rmk:points of D}
	By \cite[\href{https://stacks.math.columbia.edu/tag/0C3I}{Tag 0C3I}]{stacks-project}, the subscheme $\overline{\Dcal}^{\rm ram}$ is the ramification locus of $\overline{\Dcal}\to\PP(V_n)$.
	
	Moreover, due to the fact that taking the scheme theoretic image commutes with flat base change (\cite[\href{https://stacks.math.columbia.edu/tag/081I}{Tag 081I}]{stacks-project},) we deduce that $D$ is the projection onto $\bA(2,2)_{[3,1]}$ of the ramification locus of $\Dcal\to\bA(2,2)_{[3,1]}$.
	
	We can think of the points of $\ol{D}$ as those pairs $(q,[f])$, where $q$ is a ternary quadratic form defined up to scalar multiplication and $[f]$ is the equivalence class of a ternary form of degree $n$ (see remark \ref{rmk:points of Vn}), such that closed subscheme of $\PP^2$ determined by the homogeneous ideal $I=(q,f)$ is either singular or it contains a line.
\end{rmk}

Let $D_3$ be the restriction of $D$ to $\bA(2,2)_3$: then $D_3$ is a Cartier divisor whose points can be regarded as pairs $(q,[f])$ such that the subscheme in $\PP^2$ defined by the ideal $I=(q,f)$ is singular.

As in section \ref{sec:coh inv}, let $\Delta_{1,2n}$ be the divisor of singular forms in $\PP(1,2n)$. Then we have:
\begin{cor}\label{cor:chow diagrams for D}\cite{Dil}*{cor. 2.7.(1)}
	The $\GLt$-counterpart of $\Delta_{1,2n}$ is $D_3$. We also have commutative diagrams
	$$\xymatrix {
		CH^{\PGLt}_i(\Delta_{1,2n}) \ar[r]^{i_*} \ar[d] & CH^{\PGLt}_i(\PP(1,2n)) \ar[d] \\
		CH^{\GLt}_i(D_3) \ar[r]^{i_*} & CH^{\GLt}_i(\PP(V_n)_3) }$$
	$$\xymatrix {
		A^{\PGLt}_i(\Delta_{1,2n}) \ar[r]^{i_*} \ar[d] & A^{\PGLt}_i(\PP(1,2n)) \ar[d] \\
		A^{\GLt}_i(D_3) \ar[r]^{i_*} & A^{\GLt}_i(\PP(V_n)_3) }$$
	where the vertical arrows are all isomorphisms.
\end{cor}
\section{The geometry of the fundamental divisor}\label{sec:on the divisor}
In this section we will study the geometry of $D_2$, the restriction of the fundamental divisor $D$ (see definition \ref{def:D}) to $\PP(V_n)_2$, and we will show that $D_2$ has two irreducible components (proposition \ref{prop:Dnod has two irreducible components}).

We will also consider the proper transform $D'_{[2,1]}$ of $D_{[2,1]}$ inside a certain variety $\wt{\PP}(V_n)_{[2,1]}$ equipped with a birational morphism to $\PP(V_n)_{[2,1]}$ (see definition \ref{df:D'1}), and we will show that the restriction of $D'_{[2,1]}$ to the exceptional locus of this morphism has two irreducible components (proposition \ref{pr:DtildeE has two irr comp}).
\subsection{Projectivization of quotient vector bundles}
In this subsection we prove a general result on quotient vector bundles that we will need in the remainder of the paper.
The following proposition is a generalization of the classical construction of the projection morphism from a projective subspace.  
\begin{prop}\label{pr:quotient bundle}
	Let $X$ be a scheme of finite type over $k_0$ and suppose we have an exact sequence of locally free sheaves over $X$ of the form:
	\[ 0\longrightarrow \Fcal \longrightarrow \Ecal \longrightarrow \Gcal \longrightarrow 0 \]
	We can form the projective bundle $\pi:\PP(\Gcal)\to X$. 
	
	Consider the flat morphism $f:\PP(\Ecal)\setminus \PP(\Fcal)\to X$ obtained by composing the open embedding of $\PP(\Ecal)\setminus\PP(\Fcal)$ into $\PP(\Ecal)$ with the projection of this projective bundle onto $X$. 
	
	Then there exists a canonical morphism $\beta:\PP(\Ecal)\setminus\PP(\Fcal)\to\PP(\Gcal)$. Moreover, if there is a section $\sigma:\Ecal\to\Fcal$, the scheme $\PP(\Ecal)\setminus\PP(\Fcal)$ is isomorphic to the vector bundle over $\PP(\Gcal)$ associated to the locally free sheaf $\pi^*\Fcal\otimes\OO_{\PP(\Gcal)}(1)$.
\end{prop}
\begin{proof}
	To define a morphism $\beta:\PP(\Ecal)\setminus\PP(\Fcal)\to\PP(\Gcal)$ over $X$ we have to find an invertible sheaf $\Lcal$ over $\PP(\Ecal)\setminus\PP(\Fcal)$ and an embedding of locally free sheaves $\Lcal\hookrightarrow f^*\Gcal$. By an embedding of locally free sheaves we mean an injective morphism of coherent sheaves whose induced morphism between the fibres over the points remains injective.
	
	Let $\rho:\PP(\Ecal)\to X$ be the projection morphism. We have a canonical inclusion $\OO_{\PP(\Ecal)}(-1)\hookrightarrow \rho^*\Ecal$. Restricting this morphism to the open subscheme $\PP(\Ecal)\setminus\PP(\Fcal)$ and composing it with the morphism $f^*\Ecal\to f^*\Gcal$, we get:
	\[ \OO_{\PP(\Ecal)}(-1)|_{\PP(\Ecal)\setminus\PP(\Fcal)} \longrightarrow f^*\Gcal \]
	This last morphism is injective because the image of $\OO_{\PP(\Ecal)}(-1)|_{\PP(\Ecal)\setminus\PP(\Fcal)}$ inside $f^*\Ecal$ intersects the image of $f^*\Fcal$ only in the zero section. 
	
	This last property also holds true when passing to the fibres over the points of $\PP(\Ecal)\setminus\PP(\Fcal)$, because both $f^*\Fcal\to f^*\Ecal$ and $\OO_{\PP(\Ecal)}(-1)|_{\PP(\Ecal)\setminus\PP(\Fcal)}\to f^*\Ecal$ are embeddings of locally free sheaves, from which we also deduce that $\OO_{\PP(\Ecal)}(-1)|_{\PP(\Ecal)\setminus\PP(\Fcal)}\to f^*\Gcal$ is an emedding of locally free sheaves.
	
	To prove the second claim of the proposition, it is enough to show that on the site of $\PP(\Gcal)$-schemes there is an isomorphism of sheaves:
	\[ {\rm Hom}(-,\PP(\Ecal)\setminus\PP(\Fcal)) \simeq \pi^*\Fcal\otimes\OO_{\PP(\Gcal)}(1) \]
	Take a $\PP(\Gcal)$-scheme $S$ and a morphism $\beta:S\to\PP(\Ecal)\setminus\PP(\Fcal)$ over $\PP(\Gcal)$: due to the fact that $\PP(\Gcal)$ is a scheme over $X$, the morphism $\beta$ can be regarded as a morphism of $X$-schemes,  and it induces an inclusion
	\[ \OO_{\PP(\Ecal)}(-1)|_S\hookrightarrow \Ecal|_S \]
	The fact that $\beta$ is a morphism of $\PP(\Gcal)$-schemes implies that:
	\begin{enumerate}
		\item $\OO_{\PP(\Ecal)}(-1)|_S\simeq \OO_{\PP(\Gcal)}(-1)|_S$.
		\item After identifying $\OO_{\PP(\Ecal)}(-1)|_S$ with $\OO_{\PP(\Gcal)}(-1)|_S$, the composition of $\beta$ with the morphism $\Ecal|_S\to\Gcal|_S$ is equal to the canonical inclusion $\OO_{\PP(\Gcal)}(-1)|_S\hookrightarrow \Gcal|_S$.
	\end{enumerate}
	Therefore, the sheaf ${\rm Hom}(-,\PP(\Ecal)\setminus\PP(\Fcal))$ on the site of $\PP(\Gcal)$-schemes is isomorphic to the sheaf
	\[ {\bf F}: S \longmapsto \left\{ \OO_{\PP(\Gcal)}(-1)|_S \hookrightarrow \Ecal|_S \text{ that commutes with the morphisms to } \Gcal|_S \right\} \]
	Using the section $\sigma:\Ecal\to\Fcal$ we can construct an isomorphism ${\bf F}\simeq \pi^*\Fcal\otimes\OO_{\PP(\Gcal)}(1)$ as follows: in one direction, we send an inclusion $i:\OO_{\PP(\Gcal)}(-1)|_S \hookrightarrow \Ecal|_S$ to $\sigma\circ i$, and then we twits by $\OO_{\PP(\Gcal)}(-1)|_S$. In this way we get a morphism
	\[ \OO_S \longrightarrow \Fcal|_S \otimes \OO_{\PP(\Gcal)}(1)|_S \]
	
	In the other direction, given a morphism $\OO_S \to \Fcal|_S \otimes \OO_{\PP(\Gcal)}(1)|_S$ we first twist it by $\OO_{\PP(\Gcal)}(-1)|_S$, obtaining in this way a morphism $\alpha:\OO_{\PP(\Gcal)}(-1)|_S\to \Fcal|_S$.
	
	After that, we take the unique lifting $i:\OO_{\PP(\Gcal)}(-1)|_S\to \Ecal|_S$ of the canonical inclusion $\OO_{\PP(\Gcal)}(-1)|_S\hookrightarrow \Gcal|_S$ having the property that $\sigma\circ i=\alpha$.
	
	This construction gives us a well defined isomorphism of sheaves ${\bf F}\simeq \pi^*\Fcal\otimes\OO_{\PP(\Gcal)}(1)$.
\end{proof}

\subsection{The geometry of $D_2$}
Recall definition \ref{def:Vn}, where we introduced the vector bundles $\ol{V}_n$ and $V_n$ defined respectively over $\PP(2,2)$ and over $\bA(2,2)_{[3,1]}$, the scheme of non-zero quadratic forms in three variables. The points of $\PP(\ol{V}_n)$ can be thought as pairs $(q,[f])$, where $q$ is the projective equivalence class of a non-zero quadratic ternary form and $[f]$ is the equivalence class of a non-zero ternary form of degree $n$ (see remark \ref{rmk:points of Vn}).

Let $\ol{D}_2$ denote the pullback of $\ol{D}\to\PP(2,2)$ to the subscheme of rank $2$ conics $\PP(2,2)_2$: then it follows from the last lines of remark \ref{rmk:points of D} that the points of $\ol{D}_2$ are pairs $(q,[f])$ where $q=l_1l_2$ is a quadratic form of rank $2$ (well defined up to multiplication by a non-zero scalar) and the subscheme defined by the homogeneous ideal $I=(l_1l_2,f)$ of $\PP^2$ is either singular or has an irreducible component of dimension $1$ (this happens when one of the linear factors of $q$ divides $f$).

\begin{df}\label{df:D1}
	\hspace{1pt}
	\begin{enumerate}
		\item We define the subset $\ol{D}_2^1\subset \PP(\ol{V}_n)_2$ as the set of points $(q,[f])$ such that $q=l_1l_2$, where $l_1$ and $l_2$ are two distinct linear forms, and the closed subscheme of $\PP^2$ defined by the ideal $I=(l_1,l_2,f)$ is non-empty.
		\item We define the subset $D_2^1$ as the preimage of $\ol{D}_2^1$ along the $\Gm$-torsor $\PP(V_n)_2\to \PP(\ol{V}_n)_2$.
	\end{enumerate}
\end{df}
Thanks to the next proposition, we can make the subsets $\ol{D}_2^1$ and $D_2^1$ into closed subschemes.
\begin{prop}\label{pr:D1 irr}
	There is a scheme structure on $\ol{D}_2^1$ (resp. on $D_2^1$) which turns it into an irreducible divisor inside $\PP(\ol{V}_n)_2$ (resp. $\PP(V_n)_2$).
	Moreover, we have that $\ol{D}_2^1$ (resp. $D_2^1$) is an irreducible component of $\ol{D}_2$ (resp. $D_2$).
\end{prop}

\begin{proof}
	It is enough to show that $\ol{D}_2^1$ is an irreducible divisor inside $\PP(\ol{V}_n)_2$.
	
	Let $\Ycal^1$ be the subscheme of $\PP(2,1)\times\PP(2,1)\times\PP(2,n)\times\PP^2$ defined as follows:
	\[ \Ycal^1=\left\{ (l_1,l_2,f,p) \text{ such that } l_1(p)=l_2(p)=f(p)=0 \right\}\]
	Observe that $\Ycal^1\to\PP^2$ is a projective bundle over $\PP^2$, hence it is irreducible, and its fibres have codimension $3$ in $\PP(2,1)\times\PP(2,1)\times\PP(2,n)$. 
	
	Call $Y^1$ the image of the projection morphism $\pr_{123}:\Ycal^1\to\PP(2,1)\times\PP(2,1)\times\PP(2,n)$, where by image we mean the smallest closed subscheme of $\PP(2,1)\times\PP(2,1)\times\PP(2,n)$ through which $\pr_{123}$ factors (see \cite[\href{https://stacks.math.columbia.edu/tag/01R7}{Tag 01R7}]{stacks-project}).
	
	The morphism $\Ycal^1\to Y^1$ is generically finite, hence $Y^1$ is closed, irreducible and it has codimension $1$. Moreover, due to the properness of $\pr_{123}$, the points of $Y^1$ correspond to triples $(l_1,l_2,f)$ such that the closed subscheme of $\PP^2$ defined by the ideal $I=(l_1,l_2,f)$ is not empty.
	
	Consider the projective morphism:
	\[\pi:\PP(2,1)\times\PP(2,1)\times\PP(2,n)\longrightarrow \PP(2,2)\times\PP(2,n)\]
	which sends a triple $(l_1,l_2,f)$ to $(l_1l_2,f)$. Observe that the image of $\pi$ is $\PP(2,2)_{[2,1]}\times\PP(2,n)$, and that $\pi$ is finite of degree $2$ over its image. This implies that $\pi(Y^1)$, regarded as a subscheme inside $\PP(2,2)_{[2,1]}\times\PP(2,n)$, is closed, irreducible, it has codimension $1$ and its points correspond to pairs $(l_1l_2,f)$ such that the closed subscheme of $\PP^2$ defined by the ideal $I=(l_1,l_2,f)$ is not empty.
	
	Recall that $\ol{V}_n$ is a quotient of the vector bundle associated to the locally free sheaf $\pr_{1*}\pr_2^*\OO(n)$ by the subsheaf $\pr_{1*}\pr_2^*\OO(n-2)\otimes\OO(-1)$ (see definition \ref{def:Vn}.(1)). Applying proposition \ref{pr:quotient bundle}, we obtain a morphism:
	\[(\PP(2,2)\times\PP(2,n))\setminus \im(i)\longrightarrow\PP(\ol{V}_n) \]
	where $i:\PP(2,2)\times\PP(2,n-2)\hookrightarrow\PP(2,2)\times\PP(2,n)$ is the inclusion of (trivial) projective bundles which sends a pair $(q,f)$ to $(q,qf)$.
	The quotient morphism sends a pair $(q,f)$ to $(q,[f])$, where $[f]$ denotes the equivalence class of $f$ in the quotient.
	
	This morphism restricts to the morphism:
	\[p:(\PP(2,2)_{[2,1]}\times\PP(2,n))\setminus \im(i)_{[2,1]} \longrightarrow \PP(\ol{V}_n)_{[2,1]} \]
	We can restrict $\pi(Y^1)$ to $(\PP(2,2)_{[2,1]}\times\PP(2,n))\setminus \im(i)_{[2,1]} $ and take its image via $p$, which we denote $Z^1$. By construction, we have that $Z^1$ is closed and irreducible. Moreover, by \cite[\href{https://stacks.math.columbia.edu/tag/01R8}{Tag 01R8}.(4)]{stacks-project}, it contains $\ol{D}^1_2$ as an open, dense subset.
	
	Observe that $p$ is a topological quotient: this implies that $\ol{D}^1_2$ is closed because its preimage, which is the restriction of $\pi(Y^1)$ to $(\PP(2,2)_{[2,1]}\times\PP(2,n))\setminus \im(i)_{[2,1]}$, is closed. We deduce that $\ol{D}^1_2=Z^1$, hence that $\ol{D}^1_2$ is irreducible because $Z^1$ is so.
	
	Again by proposition \ref{pr:quotient bundle}, the morphism $p$ makes $(\PP(2,2)_{[2,1]}\times\PP(2,n))\setminus \im(i)_{[2,1]}$ into a vector bundle over $\PP(\ol{V}_n)_{[2,1]}$, whose restriction over $\ol{D}_2^1$ is $\pi(Y^1)$, which has codimension $1$ in $(\PP(2,2)_{[2,1]}\times\PP(2,n))\setminus \im(i)_{[2,1]}$: from this we deduce that $\ol{D}_2^1$ has codimension $1$.
	
	Finally, we have to show that $\ol{D}_2^1\subset \ol{D}_2$: this follows from the fact that if $(l_1l_2,[f])$ is a point of $\ol{D}_2^1$, then the subscheme of $\PP^2$ defined by the ideal $I=(l_1l_2,f)$ either contains a line or it has a singular point $p$, which is the unique point such that $l_1(p)=l_2(p)=0$.
\end{proof}

\begin{df}\label{df:D2}
	\hspace{1pt}
	\begin{enumerate}
		\item We define the subset $\ol{D}_2^2\subset \PP(\ol{V}_n)_2$ as the set of points $(q,[f])$ such that $q=l_1l_2$, where $l_1$ and $l_2$ are two distinct linear forms, and there is an $i$ such that the two plane curves defined respectively by the equations $l_i=0$ and $f=0$ do not intersect transversally. This is equivalent to the condition that the subscheme of $\PP^2$ defined by the homogeneous ideal $I=(l_i,f)$ is either singular or a line.
		\item We define the subscheme $D_2^2$ as the preimage of $\ol{D}_2^2$ along the $\Gm$-torsor $\PP(V_n)_2\to \PP(\ol{V}_n)_2$.
	\end{enumerate}
\end{df}

\begin{prop}\label{pr:D2 irr}
	There is a scheme structure on $\ol{D}_2^2$ (resp. on $D_2^2$) which turns it into an irreducible divisor inside $\PP(\ol{V}_n)_2$ (resp. $\PP(V_n)_2$).
	Moreover, we have that $\ol{D}_2^2$ (resp. $D_2^2$) is an irreducible component of $\ol{D}_2$ (resp. $D_2$).
\end{prop}	

\begin{proof}
	It is enough to prove the proposition for $\ol{D}_2^2$.
	
	We will argue as in the proof of proposition \ref{pr:D1 irr}. Let $\Ycal^2$ be the subscheme of $\PP(2,1)\times\PP(2,n)\times\PP^2$ defined as follows:
	\[ \Ycal^2:=\left\{ (l,f,p) \text{ such that } l(p)=f(p)=\det _{i} J(l,f)(p)=0 \text{ for }i=1,2,3  \right\} \]
	where $\det_i J(l,f)$ denotes the $i^{\rm th}$-minor of the Jacobian matrix associated to the forms $l$ and $f$.
	
	Zariski-locally, to verify if a triple $(l,f,p)$ is in $\Ycal^2$, it is enough to check the vanishing of only one of the three minors of $J(l,f)$: for instance, if $U_z$ denotes the open subscheme of $\PP^2$ where the coordinate $z\neq 0$, then $\Ycal^2\cap \PP(2,1)\times\PP(2,n)\times U_z$ is the locus where:
	\[ l(p)=f(p)=l_x(p)f_y(p)-l_y(p)f_x(p)=0 \]
	Let $H\subset\PP(2,1)\times\PP^2$ be the universal line: then $\Ycal^2$ can be regarded as a codimension $2$ projective subbundle of $\PP(2,n)\times H$ over $H$, hence $\Ycal^2$ is irreducible.
	
	Denote $Y^2$ the image of the projection of $\Ycal^2$ on $\PP(2,1)\times\PP(2,n)$. Due to the fact that $\Ycal^2\to Y^2$ is generically finite, we obtain that $Y^2$ is irreducible, it has codimension $1$ and its points are pairs of forms $(l,f)$ such that the associated plane curves do not intersect transversally.
	
	Consider the projective morphism:
	\[\pi:\PP(2,1)\times\PP(2,1)\times\PP(2,n)\longrightarrow \PP(2,2)\times\PP(2,n)\]
	which sends a triple $(l_1,l_2,f)$ to $(l_1l_2,f)$. Then the scheme theoretic image $\pi(\PP(2,1)\times Y^2)$ will be closed, irreducible, of codimension $1$ and its points will be pairs $(l_1l_2,f)$ such that there exists an $i\in\{1,2\}$ such that the plane curves defined respectively by the equations $l_i=0$ and $f=0$ do not intersect transversally.
	
	From here the proof works exactly as the proof of proposition \ref{pr:D1 irr}: one has only to substitute $\pi(Y^1)$ with $\pi(\PP(2,1)\times Y^2)$ and $\ol{D}_2^1$ with $\ol{D}_2^2$.	
\end{proof}	
We are ready to prove the first main result of this section:
\begin{prop}\label{prop:Dnod has two irreducible components}
	The closed subscheme $\ol{D}_2$ (resp. $D_2$) of $\PP(\ol{V}_n)_2$ (resp. $\PP(V_n)_2$) has two irreducible components, both $\GLt$-invariants and of codimension $1$, which are $\ol{D}_2^1$ and $\ol{D}_2^2$ (resp. $D_2^1$ and $D_2^2$).
\end{prop}
\begin{proof}
	We prove the proposition only for $\ol{D}_2$, as the other case easily follows from this one.
	
	We know from propositions \ref{pr:D1 irr} and \ref{pr:D2 irr} that $\ol{D}_2^1$ and $\ol{D}_2^2$ are both irreducible, of codimension $1$ and are contained in $\ol{D}_2$. The $\GLt$-invariance of both divisors is easy to check.
	
	We have to show that $\ol{D}_2^1$ and $\ol{D}_2^2$ are the only irreducible components of $\ol{D}_2$. The points of $\ol{D}_2$ are pairs $(l_1l_2,[f])$  such that $l_1\neq l_2$ and the closed subscheme $Z$ of $\PP^2$ associated to the ideal $I=(l_1l_2,f)$ is either singular or it contains a line (see remark \ref{rmk:points of D}). If $Z$ contains a line, than the point $(q,[f])$ is in $\ol{D}_2^1\cap\ol{D}_2^2$.
	
	Otherwise, the subscheme $Z$ must be supported on a set of points, and it must be singular: we can either have that the support of every singular point lays on only one of the two lines defined respectively by the equations $l_1=0$ and $l_2=0$, or that there exists a singular point whose support lays on the intersection of the two lines.
	
	In the first case, there must exist a line, say the one defined by the equation $l_1=0$, such that its intersection with the plane curve $F$ defined by the equation $f=0$ is somewhere non-transversal, thus the point $(l_1l_2,[f])$ is in $\ol{D}_2^2$.
	
	In the second case, the only possibility left is that the plane curve $F$ contains the intersection point of the two lines: this implies that $(l_1l_2,[f])$ is in $\ol{D}_2^1$.
\end{proof}

Finally, we have the following simple result on $D_1$:
\begin{prop}\label{prop:Dsq=P(Vn)sq}
	We have $\ol{D}_1=\PP(\ol{V}_n)_1$ and $D_1=\PP(V_n)_1$.
\end{prop}
\begin{proof}
	All the points in $\PP(\ol{V}_n)_1$ are of the form $(l^2,f)$, hence their associated subschemes of $\PP^2$ are all singular, and by construction the points in $\ol{D}_1$ are those pairs $(l^2,f)$ such that the associated subschemes are either singular or contain a line (see remark \ref{rmk:points of D}).
	
	The same argument also works for $D_1$.
\end{proof}
\subsection{The geometry of $D'_1$}
Recall (definition \ref{df:universal conic}) that $Q\to\PP(2,2)$ is the universal conic and $\wh{Q}\to\bA(2,2)_{[3,1]}$ is the pullback of $Q$ along the $\Gm$-torsor $\bA(2,2)_{[3,1]}\to \PP(2,2)$.
\begin{df}\label{df:Qsng}
	We define $Q\sng_{[2,1]}$ (resp. $\wh{Q}\sng_{[2,1]}$) as the closed subscheme of $Q_{[2,1]}$ (resp. $\wh{Q}_{[2,1]}$, see definition \ref{df:universal conic}) whose associated sheaf of ideals is the $1^{\rm th}$-Fitting ideal of the sheaf of relative differentials of $$\Omega^1_{Q_{[2,1]}/\PP(2,2)_{[2,1]}} \text{ (resp. }\Omega^1_{\wh{Q}_{[2,1]}/\bA(2,2)_{[2,1]}}\text{)}$$	
\end{df}
Observe that both $Q\sng_{[2,1]}$ and $\wh{Q}\sng_{[2,1]}$ inherit a $\GLt$-action.
\begin{rmk}\label{rmk:Qsng}
	\hspace{1pt}
	\begin{enumerate}
		\item We have:
		\[  Q\sng_{[2,1]}=\left\{ (q,p) \text{ such that } q_x(p)=q_y(p)=q_z(p)=0 \right\} \subset \PP(2,2)\times\PP^2 \]
		If the characteristic of the base field is $\neq 2$, we deduce that $Q\sng_{[2,1]}\to\PP^2$ is a projective bundle: actually, it is a projective subbundle of $\PP(2,2)\times\PP^2$.
		\item By \cite[\href{https://stacks.math.columbia.edu/tag/0C3I}{Tag 0C3I}]{stacks-project} we have that $\wh{Q}\sng_{[2,1]}$ is the pullback of $Q\sng_{[2,1]}$ along the $\Gm$-torsor $\bA(2,2)_{[2,1]}\to\PP(2,2)_{[2,1]}$.
		
		We also know from \cite[\href{https://stacks.math.columbia.edu/tag/0C3K}{Tag 0C3K}]{stacks-project} that $Q\sng_{[2,1]}$ (resp. $\wh{Q}\sng_{[2,1]}$) is the subscheme of $Q_{[2,1]}$ (resp. $\wh{Q}\sng_{[2,1]}$) where the projection onto $\PP(2,2)_{[2,1]}$ (resp $\bA(2,2)_{[2,1]}$) is not smooth.
	\end{enumerate}	
\end{rmk}
\begin{prop}\label{prop:Qsng bir}
	Suppose that the characteristic of the base field is $\neq 2$. Then the morphism $Q\sng_{[2,1]}\to\PP(2,2)_{[2,1]}$ is an isomorphism over $\PP(2,2)_2$. Moreover, the restriction $Q\sng_1\to\PP(2,2)_1$ is a projective bundle.
	Similarly, we have that $\wh{Q}\sng_{[2,1]}\to\bA(2,2)_{[2,1]}$ is an isomorphism over $\bA(2,2)_2$ and the restriction $\wh{Q}\sng_1\to\bA(2,2)_1$ is a projective bundle.
\end{prop}
\begin{proof}
	By direct inspection of the defining equations, we see that the singular locus of $\PP(2,2)_{[2,1]}$ is $\PP(2,2)_1$, hence $\PP(2,2)_2$ is smooth.	Observe that $Q\sng\to \PP^2$ is a projective bundle, hence $Q\sng_2$ is integral. 
	
	The morphism $Q\sng_2\to\PP(2,2)_2$ is quasi-finite (the fibre is given by the unique singular point of a pair of skew lines) and proper (this follows from the properness of $Q\to \PP(2,2)$ and the fact that the singular locus is closed and commutes with base changes when the morphism is locally finite), hence finite. It is also birational, as we can easily define a section of $Q\sng_{2,\xi}\to \xi$, where $\xi$ is the generic point of $\PP(2,2)_2$.
	
	As $\PP(2,2)_2$ is smooth, hence normal, we deduce that $Q\sng_2\to\PP(2,2)_2$ is an isomorphism.
	
	The restriction $Q\sng_1$ can be described by the following equations:
	\[ Q\sng_1=\left\{(l^2,p)\text{ such that }2l(p)l_x(p)=2l(p)l_y(p)=2l(p)l_z(p)=0\right\} \subset \PP(2,2)_1\times \PP^2\]
	At least one of the partial derivatives of $l$, evaluated in $p$, is invertible, hence the subscheme can be equivalently describes as:
	\[ Q\sng_1=\left\{ (l^2,p)\text{ such that }l(p)=0 \right\} \]
	This proves that $Q\sng_1\to\PP(2,2)_1$ is a projective subbundle of $\PP(2,2)_1\times\PP^2$.
	
	The assertions on $\wh{Q}_{[2,1]}\sng$ and $\wh{Q}_1\sng$ easily follow from what we have just proved, because we have the cartesian diagrams:
	\[ \xymatrix{
		\wh{Q}\sng_{[2,1]} \ar[d] \ar[r] & Q\sng_{[2,1]} \ar[d] \\
		\bA(2,2)_{[2,1]} \ar[r] & \PP(2,2)_{[2,1]}}
	\xymatrix{
		\wh{Q}\sng_1 \ar[d] \ar[r] & Q\sng_1 \ar[d] \\
		\bA(2,2)_1 \ar[r] & \PP(2,2)_1 }\]
	This concludes the proof.
\end{proof}
We are ready to give the main definitions of this subsection.
\begin{df}\label{df:Ptilde}
	\hspace{1pt}
	\begin{enumerate}
		\item We define	$\wt{\PP}(\ol{V}_n)_{[2,1]}$ as the pullback of $\PP(\ol{V}_n)_{[2,1]}$ along the morphism $Q\sng_{[2,1]}\to \PP(2,2)_{[2,1]}$. Similarly, we define $\wt{\PP}(V_n)_{[2,1]}$ as the pullback of $\PP(V_n)_{[2,1]}$ along the morphism $\wh{Q}\sng_{[2,1]}\to\bA(2,2)_{[2,1]}$.
		
		\item We define $\wt{\PP}(\ol{V}_n)_1$ as the restriction of $\wt{\PP}(\ol{V}_n)_{[2,1]}$ over $\PP(2,2)_1$. Similarly we define $\wt{\PP}(V_n)_1$ as the restriction of $\wt{\PP}(V_n)_{[2,1]}$ to $\bA(2,2)_1$.
	\end{enumerate}
\end{df}
\begin{df}\label{df:D'1}
	\hspace{1pt}
	\begin{enumerate}
		\item We define $\ol{D}'_{[2,1]}$ (resp. $D'_{[2,1]}$) as the closure of the preimage of $\ol{D}_2$ (resp. $D_2$) inside $\wt{\PP}(\ol{V}_n)_{[2,1]}$ (resp. $\wt{\PP}(V_n)_{[2,1]}$).
		\item We define $\ol{D}'_1$ (resp. $D'_1$) as the pullback of $\ol{D}'_{[2,1]}$ (resp. $D'_{[2,1]}$) to $\wt{\PP}(\ol{V}_n)_1$ (resp. $\wt{\PP}(V_n)_1$).		
	\end{enumerate}
\end{df}

All the objects defined above inherit a $\GLt$-action or are $\GLt$-equivariant subschemes.
\begin{lm}\label{lm:Ptilde2=P2}
	The restriction $\wt{\PP}(\ol{V}_n)_2$ of $\wt{\PP}(\ol{V}_n)_{[2,1]}$ over $Q\sng_2$ is isomorphic to $\PP(\ol{V}_n)_2$.
\end{lm}
\begin{proof}
	Follows directly from proposition \ref{prop:Qsng bir}.
\end{proof}

Let $\ol{D}_2^1$ and $\ol{D}_2^2$ (resp. $D_2^1$ and $D_2^2$) be as in the definitions \ref{df:D1} and $\ref{df:D2}$. Briefly, $\ol{D}_2^1$ is the subscheme of $\PP(\ol{V}_n)_2$ whose points are pairs $(q,[f])$ with $q=l_1l_2$ a quadratic ternary form of rank $2$ and $f$ a ternary form of degree $n$ such that there exists a point $p$ in $\PP^2$ with $l_1(p)=l_2(p)=f(p)=0$. 

On the other hand, the points of $\ol{D}_2^2$ are the pairs $(q,[f])$ with $q=l_1l_2$ and such that the plane curves $l_i=0$ and $f=0$ do not intersect transversely, for some $i\in\{1,2\}$.

\begin{df}\label{df:D'1i}
	\hspace{1pt}
	\begin{enumerate}
		\item We define $\ol{D'}^i_{[2,1]}$ as the closure inside $\wt{\PP}(\ol{V}_n)_{[2,1]}$ of the preimage of $\ol{D}_2^i$ for $i=1,2$. We also define $\ol{D'}_1^i$ as the restriction of $\ol{D'}^i_{[2,1]}$ to $\wt{\PP}(\ol{V}_n)_1$.
		\item We define $D'^i_{[2,1]}$ as the closure inside $\wt{\PP}(V_n)_{[2,1]}$ of the preimage of $D_2^i$ for $i=1,2$. We also define $D'^i_1$ as the restriction of $D'^i_{[2,1]}$ to $\wt{\PP}(V_n)_1$.
	\end{enumerate}
\end{df}

\begin{prop}\label{pr:DtildeE has two irr comp}
	We have that $\ol{D'}_1$ (resp. $D'_1$) is a codimension $1$, $\GLt$-invariant closed subscheme of $\wt{\PP}(\ol{V}_n)_1$ (resp. $\wt{\PP}(V_n)_1$) with two $\GLt$-invariant irreducible components of codimension $1$, which are $\ol{D'}_1^1$ and $\ol{D'}_1^2$ (resp. $D'^1_1$ and $D'^2_1$).
\end{prop}
\begin{proof}
	It is enough to prove the proposition for $\ol{D'}_1$.
	By proposition \ref{prop:Dnod has two irreducible components} we know that $\ol{D}^1_2$ and $\ol{D}^2_2$ are the only two irreducible components of $\ol{D}_2$: due to the fact that $\wt{\PP}(\ol{V}_n)_{[2,1]}\to\PP(\ol{V}_n)_{[2,1]}$ is an isomorphism over $\PP(\ol{V}_n)_2$, we deduce that $\ol{D'}^i_{[2,1]}$, where $i=1,2$ (see definition \ref{df:D'1i}.(1)), are the only two irreducible components of $\ol{D'}_{[2,1]}$.
	
	This implies that $\ol{D'}_1$ is the union of the two closed subschemes $\ol{D'}^1_1$ and $\ol{D'}^2_1$. We have to show that these two closed subschemes are distinct, irreducible and have codimension $1$: these properties will be easy to verify once we know what are the points of $\ol{D'}_1^1$ and $\ol{D'}_1^2$.
	
	Recall from definition \ref{df:D1} that the points of $\ol{D}_2^1$ are pairs $(l_1l_2,[f])$ with $l_1$ and $l_2$ linearly-independent linear forms and such that there exists a point $p$ in $\PP^2$ with $l_1(p)=l_2(p)=f(p)=0$. Therefore, the points of the preimage $\ol{D'}_2^1$ of $\ol{D}_2^1$ inside $\wt{\PP}(\ol{V}_n)_2$ (see lemma \ref{lm:Ptilde2=P2}) are equal to triples $(l_1l_2,[f],p)$ with $l_1(p)=l_2(p)=f(p)=0$.
	
	The points of $\ol{D'}_1^1$ are the ones in $\wt{\PP}(\ol{V}_n)_1$ that belong to the closure of $\ol{D'}_2^1$: given a diagram of the following form
	\[\xymatrix{
		\Spec(K) \ar[r] \ar[d] & \ol{D'}^1_2 \ar[d]\\
		\Spec(R) \ar[r] & \wt{\PP}(\ol{V}_n)_{[2,1]}  \\
		\Spec(k_0) \ar[r] \ar[u] & \wt{\PP}(\ol{V}_n)_1 \ar[u]}\]
	with $R$ a DVR over $k_0$ and $K$ its fraction field, it is then enough to characterize the points given by the bottom arrow.
	
	This boils down to the following question: suppose we have a family of conics $Q_R\to \Spec(R)$ that generically has rank $2$ and on the central fibre has rank $1$, and suppose we have a pseudo-divisor (i.e., the vanishing locus of $[f]$) whose restriction to the generic fibre is a divisor of degree $n$ that contains $Q_K\sng$. How can we characterize the restriction of the pseudo-divisor to the central fibre? 
	
	Consider the section $p:\Spec(R)\to Q_R$ given by the closure of $Q_K\sng$ (the existence of such a section follows easily from the properness of $Q_R\to\Spec(R)$). Then the restriction of the pseudo-divisor to the central fiber will be the equivalence class of a global section of $\OO(n)|_{Q_{k_0}}$ that vanishes on $p_{k_0}:\Spec(k_0)\to Q_{k_0}$.
	
	In other terms, let $R$ be a DVR over $k_0$ with fraction field $K$, and take a morphism $\Spec(R)\to\ol{D'}_{[2,1]}^1$ such that the induced map from $\Spec(K)$ factors through $\ol{D'}^1_2$ and the closed point lands in $\wt{\PP}(\ol{V}_n)_1$: then the image of this point will be of the form $(l^2,[f],p)$ with $l(p)=f(p)=0$, and every such point can be obtained in this way. This gives us a description of the points of $\ol{D'}_1^1$.
	
	To prove irreducibility of $\ol{D'}^1_1$, we can argue as in the proof of proposition \ref{pr:D1 irr}: consider the closed subscheme 
	\[ \Ycal^1_1:=\left\{(l,f,p)\text{ such that }l(p)=f(p)=0\right\}\]
	contained in $\PP(2,1)\times\PP(2,n)\times\PP^2$. This is a projective bundle over $\PP^2$, hence irreducible. Embed $\Ycal^1_1$ inside $\PP(2,2)\times\PP(2,n)\times\PP^2$ via the morphism that sends $(l,f,p)$ to $(l^2,f,p)$, and take its image along the morphism
	\[ \PP(2,2)\times\PP(2,n)\times\PP^2\setminus (\im(i)\times\PP^2)\longrightarrow \PP(\ol{V}_n)\times\PP^2 \]
	given by proposition \ref{pr:quotient bundle}.
	The scheme-theoretic image is then irreducible and by construction it coincides with $\ol{D'}_1^1$.
	
	On the other hand, the points in $\ol{D'}_2^2$ will be triples $(l_1l_2,[f],p)$ with $l_1(p)=l_2(p)=0$ and such that the plane curve $F$ defined by the equation $f=0$ in $\PP^2$ intersects non-transversely one of the two lines (see definition \ref{df:D2}).
	
	Consider a morphism $\Spec(R)\to \ol{D'}_{[2,1]}^2$ whose generic point lands in $\ol{D'}^2_2$, and whose closed point is sent to a point in $\wt{\PP}(\ol{V}_n)_1$: then it is not hard to see that the image of the closed point will be of the form $(l^2,[f],p)$ with $l(p)=0$ and the plane curve $F$ intersecting non-transversely the line $l=0$. This gives us a description of the points of $\ol{D'}_1^2$.
	
	To prove the irreducibility of $\ol{D'}_1^2$, consider the closed subscheme of $\PP(2,1)\times\PP(2,n)\times\PP^2\times \PP^2$:
	\[\Ycal^2_1:=\left\{(l,f,p,q)\text{ such that }l(p)=l(q)=f(q)=\det_i J(l,f)(q)=0\text{ for }i=1,2,3 \right\}\]
	where $\det_i J(l,f)$ denotes the $i^{\rm th}$-minor of the Jacobian matrix associated to the forms $l$ and $f$.
	
	Let $\Ycal^2$ be the scheme introduced in proposition \ref{pr:D2 irr}, whose points are triples $(l,f,q)$ such that $l(q)=f(q)=\det J_i(l,f)(q)=0$: we proved that $\Ycal^2$ is irreducible. By contruction there is a morphism $\Ycal^2_2\to\Ycal^2$ which makes $\Ycal^2_2$ into a projective bundle over $\Ycal^2$: this implies that $\Ycal_2^2$ is irreducible.
	
	Let $Y^2_2$ be the image of $\Ycal_2^2$ along the projection $\pr_{123}:\PP(2,1)\times\PP(2,n)\times\PP^2\times\PP^2\to\PP(2,1)\times\PP(2,n)\times\PP^2$. Embed $Y^2_2$ into $\PP(2,2)\times\PP(2,n)\times\PP^2$ via the morphism that sends $(l,f,p)$ to $(l^2,f,p)$, and take its image along the morphism
	\[ \PP(2,2)\times\PP(2,n)\times\PP^2\setminus (\im(i)\times\PP^2)\longrightarrow \PP(\ol{V}_n)\times\PP^2 \]
	given by proposition \ref{pr:quotient bundle}.
	This scheme-theoretic image will then be irreducible and by construction it coincides with $\ol{D'}_1^2$.
	
	Putting all together, we have proved that $\ol{D'}_1^1\neq\ol{D'}_1^2$, and both of them are irreducible.
	
	Finally, propositions \ref{pr:D1 irr} and \ref{pr:D2 irr} tell us that both $\ol{D}_2^1$ and $\ol{D}_2^2$ have codimension $1$ in $\PP(\ol{V}_n)_2$, thus $\ol{D'}_{[2,1]}^1$ and $\ol{D'}^2_{[2,1]}$ have codimension $1$ in $\wt{\PP}(\ol{V}_n)_{[2,1]}$ by lemma \ref{lm:Ptilde2=P2}. When restricting a subvariety to another, the codimension can only decrease: therefore, both $\ol{D'}_1^1$ and $\ol{D'}_1^2$ must have codimension $1$ in $\wt{\PP}(\ol{V}_n)_1$, as they do not coincide with the whole $\wt{\PP}(\ol{V}_n)_1$.		
\end{proof}
\begin{rmk}\label{rmk:points of D'i}
	From the proof of proposition \ref{pr:DtildeE has two irr comp} we deduce an explicit description of the points of $D'^1_1$ and $D'^2_1$. More precisely, a triple $(l^2,[f],p)$ is in $D'^1_1$ if and only if $l(p)=f(p)=0$, that is if the plane curves defined by the equations $l=0$ and $f=0$ intersect in $p$ (this condition does not depend on the representative of the equivalence class $[f]$).
	
	On the other hand, the points of $D'^2_1$ are the triples such that $l(p)=0$ and the plane curves $l=0$ and $f=0$ do not intersect transversally.	
\end{rmk}
\section{Some equivariant intersection theory}\label{sec:Chow comp}
The main goal of this section is to compute the cycle classes of some schemes that we introduced in section \ref{sec:on the divisor}, namely $D^i_2$ (see definitions \ref{df:D1} and \ref{df:D2}) and $D'^i_1$ (see definitions \ref{df:D'1i}) for $i=1,2$. These results will be used to give a proof of the key lemma \ref{lm:key} in section \ref{sec:main lemma}.
\subsection{Equivariant Chow groups}
We give here a brief introduction to equivariant Chow groups. These were first introduced in the landmark papers \cite{Tot} and \cite{EG}. A nice account of this theory can be found in \cite{FV}*{sections 2-4}.

Let $X$ be a scheme of finite type over the ground field $k_0$, and let $G$ be an algebraic group acting on it. By \cite{EG}*{lemma 9}, for any $i>0$ we can always find a $G$-representation $V$ with an open subscheme $U\subset V$ such that $G$ acts freely on $U$ and the complement $V\setminus U$ has codimension $>i$.

Suppose that $X$ is equidimensional and quasi-projective, and that the action of $G$ is linearizable: by \cite{EG}*{prop. 23} there exists a $G$-torsor $X\times U\to (X\times U)/G$ in the category of schemes. We can define the $G$-equivariant Chow group of $X$ of codimension $i$ as:
\[ CH^i_G(X):=CHi((X\times U)/G)\]
This definition does not depend on the chosen representation (\cite{EG}*{prop. 1}).

Equivariant Chow groups enjoy most of the properties of standard Chow groups:
\begin{prop}\label{pr:properties chow}
	Let $X$ and $Y$ be equidimensional, quasi-projective schemes of finite type over $k_0$, endowed with a linearized $G$-action. Then we have:
	\begin{enumerate}
		\item \emph{Proper pushforward}: every $G$-equivariant, proper morphism $f:X\to Y$ induces a homomorphism of groups
		\[f_*:CH_i^G(X)\longrightarrow CH_i^G(Y)\]
		such that $f_*[V]=[k(V):k(f(V))][f(V)]$ for every subvariety $V$ ($[k(V):k(f(V))]=0$ if $\dim(V)>\dim(f(V))$).
		
		\item \emph{Flat pullback}:  every $G$-equivariant, flat morphism $f:X\to Y$ of relative constant dimension induces a homomorphism of groups
		\[f^*:CH^i_G(Y)\longrightarrow CH^i_G(X)\]
		such that $f^*[V]=[f^{-1}V]$, where the term on the right is the fundamental class of an equidimensional closed subscheme.
		
		\item \emph{Localization exact sequence}: given a closed, $G$-invariant subscheme $Z\xhookrightarrow{i} X$ whose open complement is $U\xhookrightarrow{j} X$, there exists an exact sequence
		$$CH_i^G(Z) \xrightarrow{i_*} CH_i^G(X) \xrightarrow{j^*} CH_i^G(U) \to 0 $$
		
		\item \emph{Compatibility}: given a cartesian square of $G$-schemes
		$$ \xymatrix{
			Y \ar[r]^f \ar[d]_g & X \ar[d]_{g'} \\
			Y' \ar[r]^{f'} & X' }$$
		where the horizontal morphisms are $G$-equivariant and proper, and the vertical morphisms are $G$-equivariant and flat of relative constant dimension $d$, we get a commutative diagram:
		$$ \xymatrix{
			CH_k^G(Y') \ar[r]^{f'_*} \ar[d]^{g^*} & CH_k^G(X') \ar[d]^{g'^*} \\
			CH_{k+d}^G(Y) \ar[r]^{f_*} & CH_{k+d}^G(X) } $$
		
		\item \emph{Homotopy invariance}: if $\pi:E\to X$ is a $G$-equivariant, finite rank vector bundle, then we have an isomorphism
		$$\pi^*: CH^i_G(X)\simeq CH^i_G(E) $$
		
		\item \emph{Projective bundle formula}: If $\PP(E)\to X$ is the projectivization of a $G$-equivariant , finite rank vector bundle, then for $i< {\rm rk}(E)$ we have: 
		$$ CH^i_G(\PP(E))\simeq \oplus_{j=0}^{i} CH^{j}_G(X)$$
		
		\item \emph{Ring structure}: if $X$ is smooth, then $CH_G(X)$ inherits a graded ring structure.
		
		\item \emph{Projection formula}: if $f:X\to Y$ is a $G$-equivariant proper morphism with smooth target, then there is a well-defined pullback morphism $f^*$ and a well-defined relative intersection product $f^*(-)\cdot (-)$ such that:
		\[ f_*(\xi\cdot f^*\eta)=f_*\xi\cdot \eta \]
		
		\item \emph{Gysin homomorphism}: every $G$-equivariant, local complete intersection morphism $f:X\to Y$ induces a homomorphism of groups:
		\[ f^{!}:CH^i_G(Y)\longrightarrow CH^i_G(X) \]
		such that, if $f$ is a closed embedding and $V$ is a subvariety which intersect transversally $Y$, then $f^{!}[V]=[V\cap X]$.  
	\end{enumerate}
\end{prop}
Given an equivariant vector bundle $E\to X$ of rank $r$, we can define the \emph{equivariant Chern classes} of $E$ as homomorphisms:
\[ c_i^G(E): CH^k_G(X) \longrightarrow CH^{k+i}_G(X) \]
where $i$ ranges from $0$ to $r$. If $X$ is equidimensional and $[X]$ denotes its cycle class, the cycles $c_1^G(E)([X])$ will also be called equivariant Chern classes of $E$.

If $X$ is smooth, we can restate the projective bundle formula \ref{pr:properties chow}.(6) by saying that, given an equivariant vector bundle $E\to X$ of rank $r$, we have:
\[ CH_G(\PP(E))=CH_G(X)[h]/(f) \]
where $h=c_1^G(\OO(1))$ and $f$ is a polynomial monic and of degree $r$.

If $f:X\to Y$ is a $G$-equivariant flat morphism between equidimensional schemes and $E\to Y$ is an equivariant vector bundle, we have:
\[ f^*c_i^G(E)=c_i^G(f^*E)\]
This enables us to formulate a projection formula which is also valid for non-flat morphism.
\begin{prop}\label{pr:proj formula for chern classes}
	Let $f:X\to Y$ be a proper morphism between $G$-schemes of finite type, and let $E\to X$ be an equivariant vector bundle. Then for any $\xi$ in $CH_i^G(X)$ and for any $k$, we have:
	\[ f_*(c_k^G(f^*E)(\xi))=c_k^G(E)(f_*\xi) \]
\end{prop}
The next proposition is of fundamental importance for computational purposes.
\begin{prop}\label{prop:EG}\cite{EG}*{prop. 6}
	Let $G$ be a special group with maximal subtorus $T$ and Weyl group $W$. Then for any $G$-scheme $X$ we have:
	\[CH_G(X)=CH_T(X)^W\]
\end{prop}
Next we list some explicit presentations of equivariant Chow rings.
\begin{prop}\label{pr:chow T and GLn}\cite{EG}*{pg. 14}
	\hspace{1pt}
	\begin{enumerate}
		\item Let $T=\Gm^{\oplus n}$ be a split torus. Then
		\[ CH_T(\Spec(k_0))=\ZZ[\lambda_1,\dots,\lambda_n] \]
		where $\lambda_i$ is the equivariant first Chern class of the $T$-representation induced by the projection $T\to\Gm$ on the $i^{\rm th}$ factor.
		\item We have:
		\[ CH_{{\rm GL}_n}(\Spec(k_0))=\ZZ[c_1,\dots,c_n]\]
		where $c_i$ is the $i^{\rm th}$ equivariant Chern class of the standard representation of ${\rm GL}_n$.
	\end{enumerate}
\end{prop}
\begin{rmk}\label{rmk:symmetric algebra}
	There is a natural inclusion $CH_{{\rm GL}_n}(X)\hookrightarrow CH_{\Gm^{\oplus n}}(X)$ which sends each $c_i$ into the elementary symmetric polynomial of degree $i$ in the variables $\lambda_1,\dots,\lambda_n$. Indeed, given a ${\rm GL}_n$-scheme $X$, let $U$ be an open subscheme of a representation such that the complement has codimension $>i$ and the geometric quotient $(X\times U)/{\rm GL}_n$ exists in the category of schemes: then there is a flat morphism $(X\times U)/\Gm^{\oplus n}\to (X\times U)/{\rm GL}_n$, and the inclusion of equivariant Chow rings above is induced by the pullback along this flat map.
\end{rmk}
Proposition \ref{prop:EG} will be frequently used as follows: set $G=\GLt$ and let $T$ be the subtorus of diagonal matrices. Suppose we have a $G$-invariant subvariety $V\subset X$: then the cycle class $[V]_T$ will be invariant with respect to the $S_3$-action on $\lambda_1$, $\lambda_2$ and $\lambda_3$, hence $[V]_T=p(\sigma_1,\sigma_2,\sigma_3)$, where $\sigma_i$ is the elementary symmetric polynomial of degree $i$ in the $\lambda_j$.
The $G$-equivariant cycle class of $V$ can then be reconstructed from the $T$-equivariant one, namely we will have $[V]_G=p(c_1,c_2,c_3)$.

Any $\GLt$-equivariant morphism between $\GLt$-schemes induces a $\GLt$-equivariant pushforward and a $T$-equivariant one. The proposition that follows establishes the relation between the images of these two morphisms.
\begin{prop}\cite{FV}*{lemma 2.1}
	Let $G$ be a special algebraic group and let $T\subset G$ be a maximal subtorus. Let $X$ be a smooth $G$-scheme and $I\subset CH_G(X)$ an ideal. Then:
	\[ I\cdot CH_T(X)\cap CH_G(X)= I \]
\end{prop}

The following lemma is useful to perform computations:
\begin{lm}\label{lm:class hypersurface}\cite{EF}*{lemma 2.4}
	Let $E_1,\dots,E_n$ be $T$-representations and let $H\subset \PP(E_1)\times\cdots\times\PP(E_n)$ be a $T$-invariant hypersurface defined by the multi-homogeneous equation $f=0$, with $f$ of multidegree $(d_1,\dots,d_n)$. Let $\chi:T\to\Gm$ be the character such that for every element $t$ of $T$ we have $t\cdot f=\chi^{-1}(t)f$. Then:
	\[ [H]=c_1^T(\OO_{\PP(E_1)}(d_1))+\dots +c_1^T(\OO_{\PP(E_n)}(d_n))+c_1(\chi) \]
	inside $CH_T^1(\PP(E_1)\times\cdots\times\PP(E_n))$, where $c_1(\chi)$ is the first Chern class of the $1$-dimensional $T$-representation associated to the character $\chi$.
\end{lm}
The following is an analogue of proposition \ref{pr:Gm torsor formula}.
\begin{prop}\label{prop:chow torsor formula}\cite{Vis}*{pg. 638}
	Let $L\to X$ be an equivariant line bundle over a $G$-scheme, and let $L^*$ be the associated $\Gm$-torsor. Then we have:
	\[ CH_G^i(L^*)=CH_G^i(X)/(c_1^G(L)(CH_G^{i-1}(X)) \]
\end{prop}
\begin{proof}[Sketch of proof]
	Applying the localization exact sequence (proposition \ref{pr:properties chow}.(3)) for the zero section $X\hookrightarrow L$ we get:
	\[ CH_G(X)\longrightarrow CH_G(L) \longrightarrow CH_G(L^*)\longrightarrow 0\]
	By homotopy invariance (proposition \ref{pr:properties chow}.(5)) we can identify $CH_G(L)$ with $CH_G(X)$. To complete the proof, we only have to observe that the induced morphism $$CH_G(X)\longrightarrow CH_G(L)\simeq CH_G(X)$$ coincides by construction with the first Chern class of $L$.
\end{proof}	
We end this part with a proposition that will be useful in the next subsection.
\begin{prop}\label{pr:pullback projection is an iso}
	Let $X$ be a scheme endowed with an action of a group $G$. Suppose we have a short exact sequence of $G$-equivariant locally free sheaves:
	\[0\to \Fcal\to \Ecal\to \Gcal\to 0\]
	Then there exists an equivariant flat morphism of equivariant projective bundles $p:\PP(\Ecal)\setminus i(\PP(\Fcal))\to \PP(\Gcal)$ and the pullback morphism $p^*$ is an isomorphism of equivariant Chow groups.
\end{prop}
\begin{proof}
	The proof is by direct computation. For the sake of simplicity, we will assume that $X$ is smooth, so to have a ring structure on its equivariant Chow groups. The proof works without this assumptions, but the notation get heavier. 
	
	The existence of the morphism $p$ follows from the first part of proposition \ref{pr:quotient bundle}.
	We have:
	\begin{align*}
		CH_G(\PP(\Gcal))&\simeq CH_{G\times \Gm}(X)/(c_g^{G\times\Gm}(\Gcal))\\
		&\simeq CH_G(X)[h_{\Gcal}]/(h_{\Gcal}^g+c_1^G(\Gcal)h_{\Gcal}^{g-1}+\dots+c_g^G(\Gcal)) 
	\end{align*} 
	where $g$ is the rank of $\Gcal$ and the action of $\Gm$ on $X$ is the trivial one. The formula above is just a rewriting of the usual projective bundle formula (the hyperplane section appears as the pullback of the generator of $CH_{\Gm}(\Spec(k_0))$). 
	
	We also have:
	\begin{align*}
		CH_G(\PP(\Ecal))&\simeq CH_{G\times \Gm}(X)/(c_e^{G\times\Gm}(\Ecal))\\
		& \simeq CH_G(X)[h_{\Ecal}]/(h_{\Ecal}^e+c_1^G(\Ecal)h_{\Ecal}^{e-1}+\dots + c_e^G(\Ecal))
	\end{align*} 
	It is pretty straightforward to prove that the image of the pushforward morphism from $CH_G(\PP(\Fcal))$ to $CH_G(\PP(\Ecal))$ is the ideal generated by the cycle class $[\PP(\Fcal)]_G$, which is equal to the $G\times\Gm$-equivariant top Chern class of $\Gcal$. As $c_{e}^{G\times\Gm}(\Ecal)$ is a multiple of $c_{g}^{G\times \Gm}(\Gcal)$, we deduce that:
	\[ CH_G(\PP(\Ecal)\setminus i(\PP(\Fcal)))\simeq CH_{G\times \Gm}(X)/(c_g^{G\times\Gm}(\Gcal))  \]
	The pullback morphism $p^*$ is a morphism of $CH_G(X)$-algebras and $p^*h_{\Gcal}=h_{\Ecal}$, therefore the pullback morphism is an isomorphism.
\end{proof}
\subsection{Cycle classes of $D_2^i$}
From now on, we will always assume the characteristic of the base field to be $\neq 2$.

We will write $\CHgl(X)_{\FF_2}$ for $\CHgl(X)\otimes\FF_2$ and we will denote $\bA(2,d)$ the affine space of forms in three variables of degree $d$. 

As before, the notation $\bA(2,2)_r$ will stand for the scheme of quadratic forms in three variables of rank $r$ and $\bA(2,2)_{[a,b]}$ for the scheme of quadratic forms of rank $r$ with $a\geq r\geq b$. A similar notation will be adopted for the subschemes of $\PP(2,2)$.

Every time we deal with a product of schemes $X_1\times\cdots\times X_n$, the notation $\pr_i$ will stand for the projection morphism on $X_i$, the morphism $\pr_{ij}$ will denote the projection morphism on $X_i\times X_j$, and so forth.

Recall that $\GLt$ acts on $\bA(2,2)_3$ via the formula:
\[ A\cdot q(x,y,z):=\det(A)q(A^{-1}(x,y,z)) \]
and that $\bA(2,2)_3$ is a $\GLt$-equivariant $\Gm$-torsor over $\PP(2,2)_3$.

The following is a simple but technical proposition in equivariant intersection theory. The reader who wants to have an idea of how this lemma is applied before diving into the details of the proof is suggested to jump to remark \ref{rmk:application technical lemma}.

\begin{prop}\label{prop:cycle classes computations}
	Suppose we have a $\GLt$-equivariant morphism $f:\ol{U}\to\PP(2,2)$ and form the cartesian diagram
	\[ \xymatrix{
		U \ar[r]^{q} \ar[d] & \ol{U} \ar[d]^{f} \\
		\bA(2,2)_3 \ar[r] & \PP(2,2)}
	\]
	Denote $\PP(V_n)_U$ the pullback of $\PP(V_n)$ (see definition \ref{def:Vn}) along $U\to \bA(2,2)_3$. Then we have:
	\begin{enumerate}
		\item There exists a closed subscheme $Y\subset U\times\PP(2,n)$ and an equivariant morphism $$p:U\times\PP(2,n)\setminus Y \longrightarrow \PP(V_n)_U$$ such that the pullback morphism $p^*$ at the level of equivariant Chow groups is an isomorphism.
		\item There exists an isomorphism 
		\[ \Phi:CH^k_{\GLt}(\PP(V_n)_U)\longrightarrow CH^k_{\GLt}(\ol{U}\times\PP(2,n))/(\pr_1^*f^*s-c_1) \]
		where $s=c_1^{\GLt}(\OO_{\PP(2,2)}(1))$ and $k<2n+1$.
		\item Let $Z\subset \PP(V_n)_U$ be a subvariety of codimension $< 2n$ and let $\ol{Z}\subset \ol{U}\times\PP(2,n)$ be a subvariety such that 
		\[ (q\times{\rm id})^{-1}(\ol{Z}|_{\ol{U}\times\PP(2,n)\setminus q(Y)})=p^{-1}(Z) \]
		Then $\Phi[Z]=[\ol{Z}]$ modulo the relation $\pr_1^*f^*s-c_1=0$.
	\end{enumerate}
\end{prop}
\begin{proof}
	Consider the diagram:
	\[ \xymatrix{
		& \PP(2,2)\times\PP^2 \ar[dr]_{\pr_1} \ar[dl]^{\pr_2} & \\
		\PP(2,2) & &\PP^2 }\]
	Recall from definition \ref{def:Vn} that we have an exact sequence of locally free sheaves on $\PP(2,2)$:
	\[ 0\longrightarrow \pr_{1*}\pr^*_2\OO(n-2)\otimes\OO_{\PP(2,2)}(-1) \longrightarrow \pr_{1*}\pr^*_2\OO(n-2) \longrightarrow \ol{\Vcal}_n \longrightarrow 0 \]
	and that $V_n$ is the vector bundle associated to the locally free sheaf $\Vcal_n$, the pullback of $\ol{\Vcal}_n$ to $\bA(2,2)_3$.
	
	From this we deduce the following exact sequence of locally free sheaves on $U$:
	\[ 0\longrightarrow \pr_{1*}\pr^*_2\OO(n-2)_U\longrightarrow \pr_{1*}\pr^*_2\OO(n-2)_U \longrightarrow \Vcal_{n_U} \longrightarrow 0 \]
	Applying proposition \ref{pr:pullback projection is an iso} we get a morphism
	\[ p:U\times\PP(2,n)\setminus (U\times\PP(2,n-2)) \longrightarrow  \PP(V_n)|_U\]
	such that $p^*$ is an isomorphism of equivariant Chow groups.
	This proves (1), with $Y:=U\times\PP(2,n-2)$.
	
	Using the localization exact sequence (proposition \ref{pr:properties chow}.(3)) we can construct the following isomorphism: 
	\[ \Phi_1:CH_{\GLt}^k(\PP(V_n)|_U)\xrightarrow{p^*} CH_{\GLt}^k(U\times\PP(2,n)\setminus Y)\xrightarrow{(j^*)^{-1}} CH^k_{\GLt}(U\times\PP(2,n)) \]
	where $j$ is the open embedding $(U\times\PP(2,n))\setminus Y \hookrightarrow U\times \PP(2,n)$ (the pullback $j^*$ is an isomorphism because the codimension of $Y$ is $2n+1$ and $k<2n+1$).
	
	Observe that $U\to \ol{U}$ is the $\Gm$-torsor associated to the equivariant invertible sheaf $f^*\OO_{\PP(2,2)}(-1)\otimes\mathbb{D}$, where $\mathbb{D}$ is the determinant representation of $\GLt$, whose equivariant first Chern class is $c_1$, the generator of $CH_{\GLt}^1$ (see proposition \ref{pr:chow T and GLn}.(2)). From this we deduce that $U\times\PP(2,n)\to\ol{U}\times\PP(2,n)$ is the $\Gm$-torsor associated to the line bundle $\pr_1^*f^*\OO_{\PP(2,2)}(-1)\otimes\mathbb{D}$.	Applying proposition \ref{prop:chow torsor formula}, we get:
	\[ \Phi_2:\CHgl^k(U\times\PP(2,n))\simeq \CHgl^k(\ol{U}\times\PP(2,n))/(f^*\pr_1^*s-c_1) \]
	where as usual $s=c_1^{\GLt}(\OO_{\PP(2,2)}(1))$. We define $\Phi:=\Phi_2\circ\Phi_1$. This proves (2).
	
	To prove (3), it is enough to observe that with those assumptions we have:
	\[ \Phi_1[Z]=(j^*)^{-1}p^*[Z]=(q\times {\rm id})^*[\ol{Z}]=\Phi_2^{-1}[\ol{Z}] \]
\end{proof}	

\begin{rmk}\label{rmk:application technical lemma}
	The proposition above is used in order to simplify the computation of certain cycle classes. 
	
	For instance, recall that in definition \ref{df:D1} we introduced the divisor $D_2^1 \subset \PP(V_n)_2$ whose points correspond to pairs $(q,[f])$ where $q=l_1l_2$ is a rank $2$ quadric and the subscheme of $\PP^2$ associated to the homogeneous ideal $I=(l_1,l_2,f)$ is non-empty. In proposition \ref{pr:D1 irr} we proved that $D_2^1$ is irreducible. Concretely, the variety $D_2^1$ parametrizes conics $Q$ of rank $2$ together with a section of $\OO_Q(n)$ that vanishes on the node.
	
	Let $D^1_{[3,2]}$ be the embedding of $D^1_2$ inside $\PP(V_n)_{[3,2]}$. Then we can use proposition \ref{prop:cycle classes computations} to reduce the computation of $D^1_{[3,2]}$ to the computation of a cycle class in the equivariant Chow ring of $\PP(2,2)\times\PP(2,n)\times\PP^2$, by applying it to the morphism $\PP(2,2)_{[3,2]}\hookrightarrow\PP(2,2)$.
	
	More precisely, consider the $\GLt$-invariant, closed subscheme of $\PP(2,2)\times\PP(2,n)\times\PP^2$ defined as:
	$$ \ol{\Zcal}^1:=\left\{ (q,f,u)\text{ such that }q_x(u)=q_y(u)=q_z(u)=f(u)=0 \right\} $$
	Let $\ol{Z}^1$ be its image via the projection on $\PP(2,2)\times\PP(2,n)$.
	
	By proposition \ref{prop:cycle classes computations}.(2) applied to the morphism $\PP(2,2)_{[3,2]}\hookrightarrow\PP(2,2)$, we get an isomorphism:
	\[ \Phi: \CHgl^2(\PP(V_n)_{[3,2]}) \simeq \CHgl^2(\PP(2,2)_{[3,2]}\times\PP(2,n))/(s-c_1) \]
	where $s=c_1^{\GLt}(\OO_{\PP(2,2)_{[3,2]}}(1))$.
	
	Moreover, by proposition \ref{prop:cycle classes computations}.(3), we have that the cycle class of the restriction of $\ol{Z}^1$ to $\PP(2,2)_{[3,2]}\times\PP(2,n)$ is equal to $\Phi[D_{[3,2]}^1]$, thus to prove the lemma it is enough to compute $[\ol{Z}^1]$ inside $CH^2_{\GLt}(\PP(2,2)_{[3,2]}\times\PP(2,n))_{\FF_2}$ and substitute $s$ with $c_1$.
	
	Observe that $\pr_{12*}[\ol{\Zcal}^1]=[\ol{Z}^1]$ because $\pr_{12}$ restricted to $\ol{\Zcal}_1$ is an isomorphism, so that we have reduced ourselves to the computation of $[\ol{\Zcal}^1]$: this latter task is rather easy to complete, because $\ol{\Zcal}^1$ turns out to be a complete intersection of $T$-invariant hypersurfaces.
\end{rmk}

\begin{lm}\label{lm:class of D1}
	We have $[D_{[3,2]}^1]=c_1h\neq 0$ in $\CHgl^2(\PP(V_n)_{[3,2]})_{\FF_2}$.
\end{lm}
\begin{proof}
	As explained in remark \ref{rmk:application technical lemma}, all we have to do is to compute $[\ol{\Zcal}^1]$ in the $\GLt$-equivariant Chow ring of $\PP(2,2)\times\PP(2,n)\times\PP^2$. We can actually reduce to the $T$-equivariant Chow ring by proposition \ref{prop:EG}.
	
	The $T$-equivariant Chow ring of $\PP(2,2)\times\PP(2,n)\times\PP^2$ is isomorphic to
	\[ \ZZ[\lambda_1,\lambda_2,\lambda_3,s,h,t]/(t^3+t^2\sigma_1(\lambda_i)+t\sigma_2(\lambda_i)+\sigma_3(\lambda_i),R_s,R_h) \]
	where $s=c_1^T(\OO_{\PP(2,2)}(1))$, $h=c_1^T(\OO_{\PP(2,n)}(1))$, $t=c_1^T(\OO_{\PP^2}(1))$, the $\lambda_i$ come from $CH_T(\Spec(k_0))$ (see proposition \ref{pr:chow T and GLn}.(1)), $\sigma_i$ is the elementary symmetric polynomial of degree $i$ in three variables and $R_s$ and $R_h$ are monic polynomials respectively in $s$ and $h$, with coefficients in $\CHgl(\Spec(k_0))$. To obtain this isomorphism, one has to apply the projective bundle formula (proposition \ref{pr:properties chow}.(6)) two times.
	
	The scheme $\ol{\Zcal}^1$ is a complete intersection of four $T$-invariant hypersurfaces, which are
	\begin{align*}
		H_1:=\{q_x(p)=0\}\text{,	} & H_2:=\{q_y(p)=0\} \\
		H_3:=\{q_z(p)=0\}\text{,	} &H_4:=\{f(p)=0\} 
	\end{align*}
	Thanks to lemma \ref{lm:class hypersurface}, we can compute the cycle class of each hypersurface and then we can use the fact that $[\ol{\Zcal}^1]_T=\prod [H_i]_T$. The final result is:
	\[ [\ol{\Zcal}^1]_T= (s+t-\lambda_1)(s+t-\lambda_2)(s+t-\lambda_3)(h+nt) \]
	Using the relation $t^3+t^2\sigma_1(\lambda_i)+t\sigma_2(\lambda_i)+\sigma_3(\lambda_i)=0$ and after tensoring with $\FF_2$, we get:
	\begin{align*}
		&[\ol{\Zcal}^1]_T=t^2(s^2+sh+s\sigma_1(\lambda_i)) +t\xi_1+\xi_2&\text{ for }n\text{ odd } \\
		&[\ol{\Zcal}^1]_T=t^2(hs) +t\xi_1+\xi_2&\text{ for }n\text{ even }
	\end{align*}
	We have a cartesian diagram:
	\[ \xymatrix{
		\PP(2,2)\times\PP(2,n)\times\PP^2\ar[r]\ar[d] & \PP^2 \ar[d] \\
		\PP(2,2)\times\PP(2,n) \ar[r] & \Spec(k_0)} \]
	Hence, using the compatibility (proposition \ref{pr:properties chow}.(4)) and the projection formula (proposition \ref{pr:properties chow}.(8)), we deduce:
	\begin{align*}
		&[\ol{Z}^1]_T=s^2+sh+s\sigma_1(\lambda_i)&\text{ for }n\text{ odd } \\
		&[\ol{Z}^1]_T=hs&\text{ for }n\text{ even }
	\end{align*}
	inside $CH^2_T(\PP(2,2)\times\PP(2,n))_{\FF_2}$. We now apply proposition \ref{prop:EG} to deduce that:
	\begin{align*}
		&[\ol{Z}^1]=s^2+sh+s c_1&\text{ for }n\text{ odd } \\
		&[\ol{Z}^1]=hs&\text{ for }n\text{ even }
	\end{align*}
	inside $\CHgl^2(\PP(2,2)\times\PP(2,n))_{\FF_2}$, because the Weyl group associated to the maximal subtorus $T\subset \GLt$ is the symmetric group.
	
	Finally, substituting $s$ with $c_1$, we get $[\ol{Z}^1]=c_1h$ mod $2$, hence $[D_{[3,2]}^1]=c_1h$. This last element is not zero because by the projective bundle formula (proposition \ref{pr:properties chow}.(6)) we have that $\CHgl^1(\bA(2,2)_{[3,2]})_{\FF_2}\cdot h=\FF_2\cdot c_1h$ is a non-zero direct summand.
\end{proof}
Recall that we defined another divisor $D_2^2$ in $\PP(V_n)_2$, whose points are the pairs $(q,[f])$ such that $q=l_1l_2$ has rank $2$ and the plane curves $l_i=0$ and $f=0$ do not intersect transversally in $\PP^2$ for some $i$ (see definition \ref{df:D2}).

By proposition \ref{pr:D2 irr} we know that $D_2^2$ is an irreducible divisor in $\PP(V_n)_2$. Let $D_{[3,2]}^2$ be its embedding in $\PP(V_n)_{[3,2]}$.
\begin{lm}\label{lm:cycle class of D2}
	We have:
	\begin{itemize}
		\item $[D^2_2]=0$ in $\CHgl^1(\PP(V_n)_2)_{\FF_2}$.
		\item $[D^2_{[3,2]}]=0$ in $\CHgl^2(\PP(V_n)_{[3,2]})_{\FF_2}$.
	\end{itemize} 
\end{lm}
\begin{proof}
	Clearly, it is enough to show that $[D^2_2]=0$ in $\CHgl^1(\PP(V_n)_2)_{\FF_2}$.
	
	Recall that in the proof of proposition \ref{pr:D2 irr} we proved that there exists a closed subscheme $Y^2\subset \PP(2,1)\times\PP(2,n)$ whose points are pairs $(l,f)$ such that the plane curves $l=0$ and $f=0$ do not intersect transversally. Consider the image of $\PP(2,1)\times Y^2$ via the morphism
	\[\pi:\PP(2,1)\times\PP(2,1)\times\PP(2,n)\longrightarrow \PP(2,2)_{[2,1]}\times\PP(2,n)\]
	which sends $(l_1,l_2,f)$ to $(l_1l_2,f)$, and let $\ol{Z}^2$ be the restriction of this image to $\PP(2,2)_2\times\PP(2,n)$.
	
	By proposition \ref{prop:cycle classes computations} applied to the morphism $f:\PP(2,2)_2\to\PP(2,2)$, we get an isomorphism:
	\[ \Phi:\CHgl^1(\PP(V_n)_2)\simeq \CHgl^1(\PP(2,2)_2\times\PP(2,n))/(f^*s-c_1) \]
	such that $\Phi [D_2^2]=[\ol{Z}^2]$ modulo $f^*s-c_1$. Hence, it is enough to show that $[\ol{Z}^2]=0$ in $\CHgl(\PP(2,2)_2\times\PP(2,n))_{\FF_2}$.
	
	We claim that $[Y^2]=0$ in $\CHgl^1(\PP(2,1)\times\PP(2,n))_{\FF_2}$. This claim implies the proposition, because the morphism $\PP(2,1)\times Y^2\to\ol{Z}^2$ is generically bijective.
	
	To compute $[Y^2]$ we will first find an explicit expression for $[\Ycal^2]$ inside $\CHgl(\PP(2,1)\times\PP(2,n)\times\PP^2)$, where:
	\[  \Ycal^2:=\left\{ (l,f,p) \text{ such that }l(p)=f(p)={\det}_i J(l,f)(p)=0 \right\}\]
	Here $\det_i J(l,f)$ stands for the determinant of the minor of the Jacobian matrix $J(l,f)$ obtained by removing the $i^{\rm th}$ column.
	
	This would enable us to compute $[Y^2]$ because $\pr_{12*}[\Ycal^2]=[Y^2]$.
	
	Observe that the scheme $\Ycal^2$ is not a complete intersection but, if we restrict to the open subscheme of $\PP^2$ where $p_2\neq 0$, then we need exactly three equations to describe the restriction over this open subscheme of $\Ycal^2$, namely $l(p)=f(p)=\det_3 J(l,f)(p)=0$. 
	
	Consider the $T$-invariant subscheme
	$$\Ycal^2_1:=\left\{ (l,f,p)\text{ such that } l(p)=f(p)={\det}_3 J(l,f)(p)=0 \right\}$$
	where $T$ is the usual subtorus of $\GLt$ made of the diagonal matrices. 
	Then we have that $\Ycal^2_1$ has two irreducible components, which are $\Ycal^2$ and the $T$-invariant subscheme
	$$ \Ycal^2_2:=\left\{ (l,f,p)\text{ such that }l(p)=f(p)=p_2=0 \right\} $$
	From this we deduce that $[\Ycal^2]_T=[\Ycal^2_1]_T-[\Ycal^2_2]_T$ in $CH_T(\PP(2,1)\times\PP(2,n)\times\PP^2)$. 
	
	Observe that $\Ycal^2_1$ is a complete intersection of three $T$-invariant hypersurfaces $H_1$, $H_2$ and $H_3$ and that $\Ycal_2^2$ is a complete intersections of three $T$-invariant hypersurfaces $H_1$, $H_2$ and $H_4$, where:
	\begin{align*}
		H_1&=\left\{ l(p)=0 \right\}\\
		H_2&=\left\{ f(p)=0 \right\}\\
		H_3&=\left\{ {\det}_3J(l,f)(p)=0 \right\}\\
		H_4&=\left\{ p_2=0 \right\}
	\end{align*}
	Hence $[\Ycal_1^2]_T=[H_1]_T[H_2]_T[H_3]_T$ and $[\Ycal^2_2]_T=[H_1]_T[H_2]_T[H_4]_T$. We have that $CH_T(\PP(2,1)\times\PP(2,n)\times\PP^2)$ is generated as a $CH_T(\Spec(k_0))$-algebra by the following elements:
	\begin{align*}
		v=\pr_1^*c_1^T(\OO_{\PP(2,1)}(1)) \\
		h=\pr_2^*c_1^T(\OO_{\PP(2,n)}(1)) \\
		t=\pr_3^*c_1^T(\OO_{\PP^2}(1))
	\end{align*}
	We can compute $[H_i]_T$ using the formula given in lemma \ref{lm:class hypersurface}. In the end we get:
	\begin{align*}
		[\Ycal_1^2]-[\Ycal_2^2]&=(s+t)(h+nt)(s+h+(n-1)t-\lambda_1-\lambda_2)\\
		&-(s+t)(h+nt)(t+\lambda_3)=\\
		&=(s+t)(h+nt)(s+h+(n-2)t-\sigma_1(\lambda_i))
	\end{align*}
	Expanding the expression above, using the identity $t^3+\sigma_1(\lambda_i)t^2+\sigma_2(\lambda_i)t+\sigma_3(\lambda_i)=0$ and after tensoring with $\FF_2$, we obtain that 
	\[ [\Ycal^2]_T=t\xi_1 + \xi_2 \]
	for some $\xi_i$ in $CH_T(\PP(2,1)\times\PP(2,n))$. The compatibility property (see proposition \ref{pr:properties chow}.(4)) applied to the diagram
	\[ \xymatrix{
		\PP(2,1)\times\PP(2,n)\times\PP^2 \ar[r] \ar[d] & \PP^2 \ar[d] \\
		\PP(2,1)\times\PP(2,n) \ar[r] & \Spec(k_0)} \]
	implies that $[Y^2]_T=\pr_{12*}[\Ycal^2]_T=0$. 
	
	By injectivity of the morphism $\CHgl(\PP(2,1)\times\PP(2,n))\hookrightarrow CH_T(\PP(2,1)\times\PP(2,n))$ (see proposition \ref{prop:EG}), we conclude that $[Y^2]=0$.
\end{proof}
\subsection{Cycle classes of $D'^i_1$}
Let $\wh{Q}\to\bA(2,2)_{[3,1]}$ be the pullback of the universal conic $Q\to\PP(2,2)$, and let $\wh{Q}^{\rm sing}$ be the closed subscheme of singular points, i.e. the scheme defined by $1^{\rm th}$-Fitting ideal of $\Omega^1_{\wh{Q}/\bA(2,2)_{[3,1]}}$.

Recall that $\wh{Q}^{\rm sing}_{[2,1]}\to\bA(2,2)_{[2,1]}$ is a birational morphism, which is an isomorphism over $\bA(2,2)_2$ (see lemma \ref{lm:Ptilde2=P2}). Moreover, the fibre $\wh{Q}^{\rm sing}_1$ over $\bA(2,2)_1$ is a projective subbundle of $\bA(2,2)_1\times\PP^2$, defined by the equation:
\[ \wh{Q}^{\rm sing}_1=\left\{(l^2,p)\text{ such that }l(p)=0 \right\} \]
Observe that all these schemes inherits a $\GLt$-action, which comes from the diagonal action of $\GLt$ on the product $\bA(2,2)\times\PP^2$.

In what follows, the projective bundle $\wh{Q}^{\rm sing}_1$ will be simply denoted $E_1$, because we think of it as an exceptional divisor.
\begin{lm}\label{lm:Chow ring of E}
	We have 
	\[\CHgl(E_1)\simeq \ZZ[c_1,c_2,c_3,t,v]/(c_1-2v,f_t,f_v)\]
	where $t$ is the restriction to $E_1$ of $c_1^{\GLt}(\OO_{\PP^2}(1))$, $f_t$ is a polynomial of degree $2$ monic in $t$ and $f_v$ is a polynomial of degree $3$ monic in $v$.
\end{lm}
\begin{proof}
	We have that $E_1\to\bA(2,2)_1$ is a projective bundle, hence applying the projection bundle formula (proposition \ref{pr:properties chow}.(6)) we get:
	$$ \CHgl(E_1)\simeq\CHgl(\bA(2,2)_1)[t]/(f_t) $$
	where $t$ is the restriction to $E_1\subset\bA(2,2)_1\times\PP^2$ of $c_1^{\GLt}(\OO_{\PP^2}(1))$.
	
	Observe that $\bA(2,2)_1$ is $\Gm$-torsor over $\PP(2,2)_1$ associated to the $\GLt$-equivariant line bundle $\OO_{\PP(2,2)}(-1)|_{\PP(2,2)_1}\otimes\mathbb{D}$, where $\mathbb{D}$ is the determinant representation of $\GLt$. Applying proposition \ref{prop:chow torsor formula} we deduce:
	\[ \CHgl(\bA(2,2)_1)\simeq \CHgl(\PP(2,2)_1)/(c_1-s) \]
	where $s=c_1^{\GLt}(\OO_{\PP(2,2)}(1)|_{\PP(2,2)_1})$.
	
	There is an equivariant isomorphism
	\[\varphi:\PP(2,1)\longrightarrow\PP(2,2)_1,\quad l\longmapsto l^2 \]
	which induces an isomorphism at the level of Chow rings:
	\[\varphi^*:\CHgl(\PP(2,2)_1)\simeq \CHgl(\PP(2,1))\simeq \ZZ[c_1,c_2,c_3,v]/(f_v) \] 
	where $v=c_1^{\GLt}(\OO_{\PP(2,1)}(1))$.
	
	Observe that $\varphi^*(\OO_{\PP(2,2)}(-1)|_{\PP(2,2)_1}\otimes\mathbb{D})=\OO_{\PP(2,1)}(-2)\otimes\mathbb{D}$, hence:
	\[ \CHgl(\PP(2,2)_1)/(s-c_1)\simeq \ZZ[c_1,c_2,c_3,v]/(f_v,c_1-2v)  \]
	This concludes the proof of the lemma.
\end{proof}
Recall from definition \ref{df:Ptilde} that $\wt{\PP}(V_n)_{[2,1]}$ denotes the pullback of $\PP(V_n)_{[2,1]}\to\bA(2,2)_{[2,1]}$ along the morphism $\wh{Q}^{\rm sing}_{[2,1]}\to\bA(2,2)_{[2,1]}$, and that $\wt{\PP}(V_n)_1$ is the restriction of $\wt{\PP}(V_n)_{[2,1]}$ to $E_1=\wh{Q}^{\rm sing}_1$.

In definition \ref{df:D'1} we also introduced the divisor $D'_{[2,1]}$, which is the closure of the preimage of $D_2$ in $\wt{\PP}(V_n)_{[2,1]}$. We also defined $D'_1$ as the restriction of $D'_{[2,1]}$ to $E_1$.

By proposition \ref{pr:DtildeE has two irr comp}, we know that the divisor $D'_1$ has two irreducible components, denoted $D'^1_1$ and $D'^2_1$: the points of $D'^1_1$ are the triples $(l^2,[f],p)$ such that $l(p)=f(p)=0$. The points of $D'^2_1$ are the triples $(l^2,[f],p)$ such that $l(p)=0$ and the plane curves $l=0$ and $f=0$ do not intersect transversally (see definition \ref{df:D'1i} and remark \ref{rmk:points of D'i}).

We want to compute the cycle classes of these two components.
\begin{lm}\label{lm:cycle class of Dtilde1}
	We have $[D'^1_1]=h+nt$ in $\CHgl^1(\wt{\PP}(V_n)_1)$, where $h$ denotes the hyperplane section of the projective bundle $\wt{\PP}(V_n)_1\to E_1$ and $t$ comes from the hyperplane section of $\PP^2$.
\end{lm}
\begin{proof}
	Consider the cartesian diagram:
	\[\xymatrix{
		E_1 \ar[d] \ar[r] & \ol{E}_1 \ar[d]^{g} \\
		\bA(2,2)_{[3,1]} \ar[r] & \PP(2,2)}
	\]
	where $\ol{E}_1=Q^{\rm sing}_1$.
	We can apply proposition \ref{prop:cycle classes computations} to the diagram above: we obtain an isomorphism
	\[ \Phi: \CHgl(\wt{\PP}(V_n)_1)\simeq \CHgl(\ol{E}_1\times\PP(2,n))/(g^*s-c_1) \]
	Let $\ol{Z'}^1\subset \ol{E}_1\times\PP(2,n)$ be the $\GLt$-invariant subvariety defined as:
	$$\ol{Z'}^1:=\left\{ (l^2,f,p)\text{ such that }l(p)=f(p)=0 \right\} $$
	Then proposition \ref{prop:cycle classes computations} also implies that $\Phi[D'^1_1]=[\ol{Z'}^1]$ modulo the relation $g^*s-c_1=0$.
	
	Consider the cartesian diagram
	\[\xymatrix{
		\varphi^*\ol{E}_1 \ar[r] \ar[d] & \ol{E}_1 \ar[d] \\
		\PP(2,1) \ar[r]^{\varphi} & \PP(2,2)_1}
	\]
	where $\varphi$ sends $[l]$ to $[l^2]$. The horizontal arrows are then equivariant isomorphisms.
	
	We have an induced isomorphism:
	\begin{align*}
		\varphi^*:\CHgl(\ol{E}_1\times\PP(2,n))/(g^*s-c_1)&\simeq \CHgl(\varphi^*\ol{E}_1\times\PP(2,n))/(\varphi^*g^*s-c_1)\\
		&\simeq \ZZ[c_1,c_2,c_3,t,v,h]/(2v-c_1,f_t,f_v,f_h)
	\end{align*} 
	where $h=c_1^{\GLt}(\OO_{\PP(2,n)}(1))$ and the other generators are as in lemma \ref{lm:Chow ring of E}.
	The isomorphism $\varphi^*$ sends $[\ol{Z'}^1]$ to $[\varphi^{-1}\ol{Z'}^1]$. To compute the cycle class of $\varphi^{-1}\ol{Z'}^1$ we consider the subscheme of $\PP(2,1)\times\PP(2,n)\times\PP^2$ defined as:
	\[ \Wcal^1=\left\{ (l,f,p) \text{ such that }f(p)=0 \right\} \]
	If $i:\varphi^*\ol{E}_1\times \PP(2,n) \hookrightarrow \PP(2,1)\times\PP(2,n)\times\PP^2$ is the closed immersion, then we have $i^{!}[\Wcal^1]=[\ol{\Zcal'}^1]$ (see proposition \ref{pr:properties chow}.(9)). We can apply lemma \ref{lm:class hypersurface} to $\Wcal_1$, which is defined by a $T$-invariant equation of tridegree $(0,1,n)$, so that we get $[\Wcal^1]=h+nt$.
	
	This implies: 
	\[[\ol{\Zcal'}^1]=i^{!}[\Wcal^1]=h+nt\]
	Here there is a little and quite common abuse of notation, as we are denoting $t$ both the pullback of $c_1^{\GLt}(\OO_{\PP^2}(1))$ to $\PP(2,1)\times\PP^2$ and its restriction to $\varphi^*\ol{E}_1$. This concludes the proof.
\end{proof}
Observe that $D'^1_1$ can also be seen as a codimension $2$ integral subscheme of $\wt{\PP}(V_n)_{[2,1]}$.
\begin{lm}\label{cor:cycle class Dtilde1 non zero}
	The class $[D'^1_1]$ is non-zero in $\CHgl^2(\wt{\PP}(V_n)_{[2,1]})_{\FF_2}$.
\end{lm}
To prove the lemma above, we need a technical result.
\begin{lm}\label{lm:E1 not zero}
	The cycle class $[E_1]$ is not zero in $\CHgl^1(\wh{Q}\sng_{[2,1]})_{\FF_2}$.
\end{lm}
\begin{proof}
	From remark \ref{rmk:Qsng}.(2) we know that there is a cartesian diagram:
	\[\xymatrix{
		\wh{Q}\sng_{[2,1]} \ar[r]^{f'} \ar[d]^{q'} & Q\sng_{[2,1]} \ar[d]^{q} \\
		\bA(2,2)_{[3,1]} \ar[r]^{f} & \PP(2,2) }\]
	Therefore, the top horizontal morphism is a $\Gm$-torsor whose associated equivariant line bundle is $q^*\OO_{\PP(2,2)}(-1)\otimes\mathbb{D}$, where the latter is the determinant representation of $\GLt$.
	
	This implies (proposition \ref{prop:chow torsor formula}) that the pullback morphism
	\[f'^*:\CHgl(Q\sng_{[2,1]})\longrightarrow\CHgl(\wh{Q}\sng_{[2,1]})\]
	is surjective with kernel the ideal $(c_1-q^*s)$. In particular, to compute $[E_1]$ we can equivalently check that $[\ol{E}_1]$ is not zero modulo the ideal $(c_1-q^*s)$, where $\ol{E}_1=Q\sng_1$. This is because we have the cartesian diagram:
	\[\xymatrix{
		E_1 \ar[r] \ar[d] & \wh{Q}\sng_{[2,1]} \ar[d] \\
		\ol{E}_1 \ar[r] & Q\sng_{[2,1]}}\]	
	The equivariant Chow ring of $Q\sng_{[2,1]}$ can be easily determined applying the projective bundle formula (proposition \ref{pr:properties chow}.(6)), because $Q\sng_{[2,1]}$ is a projective bundle over $\PP^2$ (see remark \ref{rmk:Qsng}). We have:
	\[ \CHgl(Q\sng_{[2,1]})\simeq\ZZ[c_1,c_2,c_3,s,t]/(R_s,R_t) \]
	where $t$ is the pullback of the hyperplane section of $\PP^2$ and $s$ is the restriction of the hyperplane section of $\PP(2,2)$.
	
	Proposition \ref{prop:EG} tells us that for any $\GLt$-scheme $X$ we have an inclusion of $\CHgl(X)$ inside $CH_T(X)$, where $T$ denotes the subtorus of diagonal matrices: applying this to our case, we deduce that it is enough to show that $[\ol{E}_1]_T\neq 0$ modulo the ideal $(\sigma_1(\lambda_i)-s)$, where the $\lambda_i$ are the generators of $CH_T(\Spec(k_0))$ (see proposition \ref{pr:chow T and GLn}).
	
	Let $U$ be the open, $T$-invariant subscheme $\PP^2\setminus\{[1:0:0],[0:1:0]\}$ of $\PP^2$. Denote $Q\sng_{[2,1]}|_U$ (resp. $\ol{E}_1|_U$) the restriction over $U$ of $Q\sng_{[2,1]}\to\PP^2$ (resp. $\ol{E}_1$). The localization exact sequence (proposition \ref{pr:properties chow}.(3)) implies that the restriction morphism induces an isomorphism 
	\begin{align*}
		CH^1_T(Q\sng_{[2,1]}|_U)/\langle \sigma_1(\lambda_i)-s \rangle &\simeq CH^1_T(Q\sng_{[2,1]})/\langle \sigma_1(\lambda_i)-s \rangle\\
		&\simeq \ZZ\langle \lambda_1,\lambda_2,\lambda_3,s,t\rangle\ / \langle \sigma_1(\lambda_i)-s \rangle
	\end{align*} 
	because the complement of the open subscheme $U$ in $\PP^2$ has codimension $2$.
	We have reduced ourselves to show that $[\ol{E}_1|_U]\neq 0$ modulo $\sigma_1(\lambda_i)-s$.
	
	We can write $[\ol{E}_1|_U]_T=ns+mt+\sum_{i=1}^{3} k_i\lambda_i$, where $n$, $m$ and $k_i$ are coefficients in $\FF_2$. Observe that it is enough to prove $m\neq 0$. If $p:\ol{E}_1|_U\to U$ is the usual projection, then we have:
	\begin{align*} 
		p_*([\ol{E}_1|_U]\cdot s^2)&=np_*s^3+\sum_{i=1}^{3}k_i\lambda_i\cdot p_*s^2+mt\cdot p_*s^2=\xi+mt
	\end{align*}
	where $\xi$ is linear polynomial in the $\lambda_i$. In the identity above we used several facts: the projection formula (proposition \ref{pr:properties chow}.(8)) to obtain $mp_*(t\cdot s^2)=mt\cdot p_*s^2$, the fact that for every projective bundle $p:\PP(E)\to X$ whose fibres have dimension $d$ we have $p_*c_1(O_E(1))^d=[X]$ (see \cite{Ful}*{prop. 3.1.(a).ii}) and finally that the relation $R_s=0$ is a monic polynomial of degree $3$ in $s$ with coefficients in $CH_T(\Spec(k_0))$.
	If we prove that $m\neq 0$, we are done, because $\{\lambda_1,\lambda_2,\lambda_3,t\}$ is a basis for $CH^1_T(\PP^2)_{\FF_2}$.
	
	Let $H_1$ be the $T$-invariant hyperplane in the projective bundle $\PP(2,2)\times U\to U$ whose points are those pairs $(q,p)$ such that $q(1,0,0)=0$. Similarly, define $H_2$ as the hyperplane whose points are pairs $(q,p)$ such that $q(0,1,0)=0$. The equations defining these two hyperplanes are respectively $q_{11}=0$ and $q_{22}=0$, where $q=q_{11}x^2+q_{22}y^2+\dots$.
	
	Lemma \ref{lm:class hypersurface}, together with the description above of $H_i$, tells us that $[H_i]_T=c_1^{T}(\OO_{\PP(2,2)}(1))+2\lambda_i$. Observe that 
	\[i:Q\sng_{[2,1]}|_U\hookrightarrow \PP(2,2)\times U \]
	is a regular embedding because both the domain and the target are smooth, hence there is a well defined Gysin homomorphism (proposition \ref{pr:properties chow}.(9)):
	\[ i^!:CH_T(\PP(2,2)\times U)\longrightarrow CH_T(Q\sng_{[2,1]}|_U) \]
	In particular, $i^!([H_1]_T\cdot [H_2]_T)=s^2$ in $CH_T(Q\sng_{[2,1]}|_U)_{\FF^2}$. On the other hand, we have that $i^!([H_1]_T[H_2]_T)=[H_1\cap H_2\cap Q\sng_{[2,1]}|_U]_T$, where:
	\[ H_1\cap H_2 \cap Q\sng_{[2,1]}|_U=\left\{ \begin{matrix}
	(q,p) \text{ such that } q_x(p)=q_y(p)=0\\
	q_z(p)=q(1,0,0)=q(0,1,0)=0 \\
	\text{ and }p\neq[1:0:0],[0:1:0]
	\end{matrix}  \right\} \]
	We can write $p_*[\ol{E}_1]_T\cdot s^2=p_*[\ol{E}_1\cap H_1\cap H_2]_T$, and $\ol{E}_1\cap H_1\cap H_2$ is sent by $p$ onto the restriction to $U$ of the unique line $L=\{z=0\}$ that passes through $(1,0,0)$ and $(0,1,0)$. Moreover, $\ol{E}_1\cap H_1\cap H_2$ is 1:1 on $L$, therefore
	\[ p_*[\ol{E}_1\cap H_1\cap H_2]_T=[L]_T=t+\lambda_3 \]
	This proves that $[\ol{E}_1|_U]_T=t+\xi$ in $CH_T(Q\sng_{[2,1]}|_U)_{\FF_2}$, where $\xi$ is a linear combination of $\lambda_i$ and $s$, and it concludes the proof.	
\end{proof}
Now we are ready to give a proof of lemma \ref{cor:cycle class Dtilde1 non zero}.
\begin{proof}[Proof of lemma \ref{cor:cycle class Dtilde1 non zero}]
	From lemma \ref{lm:cycle class of Dtilde1} we know that $[D'^1_1]=h+nt$ in $\CHgl^1(\wt{\PP}(V_n)_1)$, where $h$ is the hyperplane section of the projective bundle $\wt{\PP}(V_n)_1\to E_1$ and $t$ is the pullback of the hyperplane section of $\PP^2$ along the morphism $\wt{\PP}(V_n)_1\to E_1\to\PP^2$.
	
	Let $j:E_1\to\wh{Q}\sng_{[2,1]}$ be the closed immersion, so that we have a cartesian square:
	$$\xymatrix{
		\wt{\PP}(V_n)_1 \ar[r]^{j'} \ar[d]^{p'} & \wt{\PP}(V_n)_{[2,1]} \ar[d]^{p} \\
		E_1 \ar[r]^{j} & \wh{Q}^{\rm sing}_{[2,1]} }$$
	We want to prove that $j'_*[D'^1_1]$ is non-zero in $\CHgl^2(\wt{\PP}(V_n)_{[2,1]})_{\FF^2}$.
	
	We have:
	\begin{align*}
		j'_*[D'^1_1]&=j'_*h+nj'_*t=j'_*(h\cdot p'^*1)+j'_*p'^*c_1^{\GLt}(\OO_{\PP^2}(1))\\
		&=h\cdot p^*j_*1+p^*j_*c_1^{\GLt}(\OO_{\PP^2}(1))
	\end{align*}
	where in the second identity we applied the projection formula for Chern classes (proposition \ref{pr:proj formula for chern classes}) to the proper morphism $j'$: with a common abuse of notation, we used the symbol $h$ to denote both the hyperplane section of $\wt{\PP}(V_n)_{[2,1]}$ and its pullback along $j$.
	
	The third identity is a consequence of the compatibility property (proposition \ref{pr:properties chow}.(4)).
	
	The projective bundle formula (proposition \ref{pr:properties chow}.(6)) applied to $\CHgl^2(\wt{\PP}(V_n)_{[2,1]})$ tells us that this group decomposes as a direct sum
	\begin{align*}
		\CHgl^2(\wt{\PP}(V_n)_{[2,1]})&\simeq p^*\CHgl^2(\wt{\bA}(2,2)_{[2,1]})\oplus p^*\CHgl^1(\wt{\bA}(2,2)_{[2,1]})\cdot h \\
		&\oplus p^*\CHgl^0(\wt{\bA}(2,2)_{[2,1]})\cdot h^2
	\end{align*}
	Hence it is enough to check that $[E_1]\neq 0$ in $\CHgl^1(\wh{Q}^{\rm sing}_{[2,1]})$, which follows from lemma \ref{lm:E1 not zero}.
\end{proof}

We now focus on $D'^2_1$ (definition \ref{df:D'1i}). Recall that the points of $D'^2_1$ can be seen as triples $(l^2,[f],p)$ such that the plane curves $l=0$ and $f=0$ does not intersect transversely at $p$. This property is clearly independent of the choice of a representative of $[f]$.

Observe that, just as for $D'^1_1$, this scheme can be seen both as a codimension $1$ subvariety of $\wt{\PP}(V_n)_1$ and as a codimension $2$ subvariety of $\wt{\PP}(V_n)_{[2,1]}$.
\begin{lm}\label{lm:cycle class of Dtilde2}
	We have $[D'^2_1]=0$ in $\CHgl^1(\wt{\PP}(V_n)_1)_{\FF_2}$ and $\CHgl^2(\wt{\PP}(V_n)_{[2,1]})_{\FF_2}$.
\end{lm}
\begin{proof}
	Clearly, the second assertion follows from the first one.
	Applying proposition \ref{prop:cycle classes computations} exactly in the same way as in the proof of lemma \ref{lm:cycle class of Dtilde1}, we see that it is enough to show $[\ol{Z'}^2]=0$ in $\CHgl^1(\ol{E}_1\times\PP(2,n))/(g^*s-c_1)$, where $g$ is the obvious morphism to $\PP(2,2)$ and $\ol{Z'}^2$ is the subvariety whose points are triples $(l^2,f,p)$ such that the plane curves $l=0$ and $f=0$ intersect non transversally in $p$.
	
	As in the proof of lemma \ref{lm:cycle class of Dtilde1}, we have a cartesian diagram 
	\[\xymatrix{
		\varphi^*\ol{E}_1 \ar[r] \ar[d] & \ol{E}_1 \ar[d] \\
		\PP(2,1) \ar[r]^{\varphi} & \PP(2,2)_1}
	\]
	where $\varphi$ sends $[l]$ to $[l^2]$, and the horizontal arrows are equivariant isomorphisms.
	
	Therefore, there is an equivariant isomorphism $\psi:\varphi^*\ol{E}_1\times\PP(2,n)\simeq \ol{E}_1\times\PP(2,n)$. After identifying the equivariant Chow rings of these two schemes, we have that $[\ol{Z'}^2]=[\psi^{-1}\ol{Z'}^2]$.
	
	Recall that in the proof of lemma \ref{lm:cycle class of D2} we introduced the subscheme $Y^2$ of $\PP(2,1)\times\PP(2,n)$ whose points are pairs $(l,f)$ such that the plane curves $l=0$ and $f=0$ do not intersect transversally, and we also proved that $[Y^2]=0$ in $\CHgl^1(\PP(2,1)\times\PP(2,n))_{\FF_2}$.
	
	Observe that the embedding $i:\varphi^*\ol{E}_1\times\PP(2,n)\hookrightarrow \PP(2,1)\times\PP(2,n)\times\PP^2$ is a local complete intersection, because both schemes are regular. We can apply proposition \ref{pr:properties chow}.(9), which tells us that there is a well defined Gysin homomorphism:
	$$ i^{!}:\CHgl(\PP(2,n)\times \PP(2,1)\times\PP^2)\longrightarrow\CHgl(\varphi^*\ol{E}_1\times\PP(2,n)) $$
	In particular, we have $i^{!}[Y^2\times\PP^2]=[\psi^{-1}\ol{Z'}^2]$, thus the latter cycle class is zero in $\CHgl(\varphi^*\ol{E}_1\times \PP(2,n))_{\FF_2}$.
\end{proof}
\section{The key lemma}\label{sec:main lemma}
In this section, we will always be working with Chow groups with coefficients in $\FF_2$.
The goal is to prove the key lemma \ref{lm:key}, which is the only missing ingredient for completing the computation of the cohomological invariants of $\Hcal_g$ (see section \ref{sec:coh inv}). Let us restate here what we want to prove:
\begin{kl-no-num}
	Let $n\geq 1$ be an integer and let $k_0$ be an algebraically closed field of characteristic $\neq 2$. Let $i:\Delta_{1,2n}\hookrightarrow\PP(1,2n)$ be the inclusion of the subscheme of singular forms into the projective space of binary forms of degree $2n$ over $k_0$. Then the pushforward homomorphism:
	\[i_*: \Apgl^0(\Delta_{1,2n},\Het)\longrightarrow\Apgl^1(\PP(1,2n),\Het)\]
	between equivariant Chow groups with coefficients in $\Het:=\oplus_i H^i_{\'{e}t}(-,\mu_2^{\otimes i})$ vanishes.
\end{kl-no-num}
From now on, we assume that the characteristic of the base field $k_0$ is $\neq 2$.
\subsection{Proof of the key lemma}
Corollary \ref{cor:chow diagrams for D} implies that $\Apgl^0(\Delta_{1,2n})\to \Apgl^1(\PP(1,2n))$ is zero if and only if the morphism
\[ i_*:\Agl^0(D_3)\longrightarrow \Agl^1(\PP(V_n)_3) \]
is zero.

if $\wh{Q}_3$ denotes the pullback to $\bA(2,2)_3$ of the universal smooth conic $Q_3\to\PP(2,2)_3$, then $\PP(V_n)_3$ is isomorphic to ${\rm Hilb}^{2n}_{\wh{Q}_3/\bA(2,2)_3}$ (see remark \ref{rmk:hilb}), and $D_3$ is the divisor parametrizing non-\'{e}tale subschemes.

\begin{rmk}
	Before proceeding with a detailed proof of the key lemma, let us briefly sketch our strategy. It is easy to see that $i_*$ is zero in cohomological degree $0$ (see lemma \ref{lm: i_sm is zero in degree 0}). Unfortunately, there are elements of higher cohomological degree contained in $\Agl^1(\PP(V_n)_3)$, namely the ones coming from $\Bcal\PGLt$, so we cannot conclude from here that $i_*$ is zero. The ramification of such elements will be supported on $\PP(V_n)_2=p^{-1}(\bA(2,2)_2)$, i.e. the locus formed by the points parametrizing singular conics.
	
	On the other hand, we expect the ramification of the elements coming from $D_3$ to be related to the stratification of the projective bundle $\PP(V_n)_{[3,2]}$ determined by the type of singularity of the support of the relative section of degree $2n$, no matter what the rank of the conic is. In particular, this would imply that their pushforward to $\PP(V_n)_3$ cannot coincide with the pullbacks of elements from $\Bcal\PGLt$, hence the image of $i_*$ must be zero.
	
	The advantage of using the $\GLt$-quotients rather than the $\PGLt$-quotients is that the former enable us to see the ramification of the invariants coming from $\Bcal\PGLt$.
\end{rmk}

Let us now formalize the ideas contained in the remark above. We know almost nothing about $\Agl^0(D_3)$ but, on the other side, we know a lot about $\Agl^1(\PP(V_n)_3)$. Indeed, from the formula for equivariant Chow rings with coefficients of projective bundles (proposition \ref{pr:properties}.(7)), we have
$$ \Agl^1(\PP(V_n)_3)\simeq\Agl^1(\bA(2,2)_3)\oplus\Agl^0(\bA(2,2)_3)\cdot h$$
where $h$ is equal to $c_1^{\GLt}(\OO(1))$, which is an element of codimension $1$ and degree $0$. Proposition \ref{prop: chow diagrams} applied to the $\PGLt$-scheme $\Spec(k_0)$, whose $\GLt$-counterpart is $\bA(2,2)_3$, tells us that $\Apgl(\Spec(k_0))\simeq \Agl(\bA(2,2)_3)$. Combining this with proposition \ref{pr:Rost of BPGL_2 and P1}, we readily deduce that 
$$ \Agl^1(\PP(V_n)_3)\simeq(\CHgl^1(\PP(V_n)_3)_{\FF_2})\oplus \FF_2\cdot \tau_1 \oplus \FF_2\cdot w_2h $$
where the first addend coincides with the elements of degree $0$, the element $\tau_1$ has codimension and degree both equal to $1$, and finally $w_2$ has codimension $0$ and degree $2$, so that $w_2h$ has codimension $1$ and degree $2$. 

Thanks to the fact that $i_*$ preserves the degree (proposition \ref{pr:properties}.(2)), every element in $\Agl^0(D_3)$ of degree greater than $2$ will be sent by $i_*$ to $0$. We need to find out if there is any element $\alpha$ of degree smaller or equal to $2$ such that $i_*\alpha$ is not zero.
\begin{lm}\label{lm: i_sm is zero in degree 0}
	The morphism $i_*:\CHgl^0(D_3)_{\FF_2}\to\CHgl^1(\PP(V_n)_3)_{\FF_2}$ is zero.
\end{lm}
\begin{proof}
	We have to show that the cycle class $[D_3]=0$ in $\CHgl^1(\PP(V_n)_3)_{\FF_2}$. From \cite[proposition 4.2]{Dil} we have that $[D_3]=4(n-2)h$ in $\CHgl(\PP(V_n)_3)$, where $h=c_1^{\GLt}(\OO_{\PP(V_n)_3}(1))$. This implies the lemma.
\end{proof}
We have a closed immersion of $\PP(V_n)_2$ inside $\PP(V_n)_{[3,2]}$, whose open complement is $\PP(V_n)_3$. Moreover, if we pullback $D_{[3,2]}$ along this closed immersion, we obtain $D_2$. By compatibility (proposition \ref{pr:properties}.(5)) we get:
\begin{equation} \label{diag nod}
	\xymatrix{
		\Agl^0(D_{[3,2]}) \ar[r]^{i^{[3,2]}_*} \ar[d]^{j_D^*} & \Agl^1(\PP(V_n)_{[3,2]}) \ar[d]^{j^*} \\
		\Agl^0(D_3) \ar[r]^{i_*} \ar[d]^{\del_D} & \Agl^1(\PP(V_n)_3) \ar[d]^{\del} \\
		\Agl^0(D_2) \ar[r]^{i^{2}_*} \ar[d]^{f_{D*}} & \Agl^1(\PP(V_n)_2) \ar[d]^{f_*} \\
		\Agl^1(D_{[3,2]}) \ar[r]^{i^{[3,2]}_*}  & \Agl^2(\PP(V_n)_{[3,2]}) }
\end{equation}
Observe that the vertical sequences are exact.
\begin{lm}\label{cor: not zero iff not zero for nod}
	The Chow group with coefficients $\Agl^1(\PP(V_n)_{[3,2]})$ is concentrated in degree $0$. In particular, if $\alpha$ is an element of $\Agl^0(D_3)$ of degree greater than $0$, then $i_*\alpha=0$ if and only if $i^{2}_*(\del_D\alpha)=0$.
\end{lm}
\begin{proof}
	Using the projective bundle formula (proposition \ref{pr:properties}.(7)) we have
	$$ \Agl^1(\PP(V_n)_{[3,2]})=\Agl^1(\bA(2,2)_{[3,2]})\oplus\Agl^0(\bA(2,2)_{[3,2]})\cdot h $$
	where $h=c_1^{\GLt}(\OO_{\PP(V_n)_{[3,2]}}(1))$.
	The scheme $\bA(2,2)_{[3,2]}$ is an open subscheme of $\bA(2,2)$ whose complement has codimension $3$. The localization exact sequence (proposition \ref{pr:properties}.(4)) implies that:
	$$ \Agl^i(\bA(2,2)_{[3,2]})=\Agl^i(\bA(2,2)) \text{ for }i=0,1 $$
	The $\GLt$-representation $\bA(2,2)$ can be regarded as a $\GLt$-equivariant vector bundle over $\Spec(k_0)$, hence by homotopy invariance (proposition \ref{pr:properties}.(6)) we obtain:
	$$ \Agl^i(\bA(2,2))=\Agl^i(\Spec(k_0)) $$
	By proposition \ref{pr:Rost of BPGL_2 and P1}, there are no non-zero elements of degree greater than $0$ in this ring, therefore $\Agl^1(\PP(V_n)_{[3,2]})$ is concentrated in degree $0$.
	
	Let $\alpha$ be an element of $\Agl^0(D_3)$ of degree greater than $0$. Then $i_*\alpha=0$ implies that $i^{^2}_*(\del_D)=0$ because of the commutativity of the middle square of diagram \ref{diag nod}. On the other hand, if $i^{^2}_*(\del_D\alpha)=0$ then the exactness of the vertical sequences of diagram \ref{diag nod} implies that $i_*\alpha=j^*\beta$ for some $\beta$ in $\Agl^1(\PP(V_n)_{[3,2]})$. 
	
	By proposition \ref{pr:properties}.(2) and \ref{pr:properties}.(3), we have $\deg(\alpha)=\deg(i_*\alpha)=\deg(\beta)$, thus $\beta=0$ because $\Agl^1(\PP(V_n)_{[3,2]})$ is concentrated in degree $0$.
\end{proof}
\begin{prop}\label{pr:i* is zero for elements of degree one}
	If $\alpha$ in $\Agl^0(D_3)$ has degree $1$, then $i_*\alpha=0$.
\end{prop}
\begin{proof}
	By hypothesis, $\del_D\alpha$ is a degree zero element of $\Agl^0(D_2)$: the degree zero part of this group can be identified with $\CHgl^0(D_2)_{\FF_2}$ (proposition \ref{pr:properties}.(1)). From proposition \ref{prop:Dnod has two irreducible components} we deduce that
	$$\CHgl^0(D_2)_{\FF_2}\simeq \CHgl^0(D_2^1)_{\FF_2}\oplus \CHgl^0(D_2^2)_{\FF_2}\simeq \FF_2\oplus\FF_2$$
	where $D^1_2$ and $D^2_2$ are the two irreducible components of $D_2$ (see definition \ref{df:D1} and \ref{df:D2}).
	
	Write $\del_D\alpha=(n,m)$. From lemma \ref{cor: not zero iff not zero for nod} we have that $i_*\alpha=0$ if and only if
	$$ 0=i^{2}_*(n,m) = n[D_2^1]+m[D_2^2] \in \CHgl^1(\PP(V_n)_2)_{\FF_2} $$
	Because of the exactness of the left vertical sequence of the diagram \ref{diag nod}, we have $f_{D*}(\del_D\alpha)=0$. This implies that 
	$$0=i^{[3,2]}_*f_{D*}(\del_D\alpha)=n[D_2^1]+m[D_2^2] $$
	in $\CHgl^2(\PP(V_n)_{[3,2]})$. By lemma \ref{lm:class of D1} we know that $[D^1_1]=c_1h \neq 0$ and lemma \ref{lm:cycle class of D2} tells us that $[D_2^2]=0$: we deduce $0=nc_1h$, thus $n=0$ and $\del_D\alpha=(0,m)$
	
	In particular $i_*^{2}\del_D\alpha=m[D^2_2]$ in $\CHgl^1(\PP(V_n)_2)$, and by lemma \ref{lm:cycle class of D2} we know that this last term is zero: this concludes the proof.
\end{proof}

In the next lines, we will prove a result similar to the one of proposition \ref{pr:i* is zero for elements of degree one} but relative to the closed embedding $D_2\subset \wt{\PP}(V_n)_2$ insted of $D_3\subset \PP(V_n)_3$ (see proposition \ref{pr:i_* is zero in degree two})

Let $\wh{Q}\sng_{[2,1]}\to \bA(2,2)_{[2,1]}$ be the singular locus of the morphism $\wh{Q}_{[2,1]}\to\bA(2,2)_{[2,1]}$, where $\wh{Q}_{[2,1]}$ is the pullback to $\bA(2,2)_{[2,1]}$ of the universal conic $Q_{[2,1]}\to\PP(2,2)_{[2,1]}$ restricted over the closed subscheme of conics of rank $< 3$ (see definition \ref{df:Qsng}). Denote $E_1$ the restriction of $\wh{Q}\sng_{[2,1]}\to \bA(2,2)_{[2,1]}$ over $\bA(2,2)_1$. Recall that $\wh{Q}\sng_2\to\bA(2,2)_2$ is an isomorphism, and that $E_1\to \bA(2,2)_1$ is a projective bundle (see proposition \ref{prop:Qsng bir}).

Let $\wt{\PP}(V_n)_{[2,1]}$ be the pullback of $\PP(V_n)_{[2,1]} \to \bA(2,2)_{[2,1]}$ along the morphism $\wh{Q}\sng_{[2,1]}\to \bA(2,2)_{[2,1]}$, and define $\wt{\PP}(V_n)_1$ as the restriction of $\wt{\PP}(V_n)_{[2,1]}$ over $E_1$ (see definition \ref{df:Ptilde}). The restriction $\wt{\PP}(V_n)_2$ of $\wt{\PP}(V_n)_{[2,1]}$ over $\wh{Q}\sng_2$ is isomorphic to $\wt{\PP}(V_n)_2$ (see lemma \ref{lm:Ptilde2=P2}).

Finally, let $D'_{[2,1]}$ be as in definition \ref{df:D'1} and let $D'_1$ be the restriction of $D'_{[2,1]}$ over $E_1$. Then by compatibility (proposition \ref{pr:properties}.(5)), we get the following commutative diagram, whose vertical sequences are exact:
\begin{equation}\label{diag tilde}
	\xymatrix{
		\Agl^0(D'_{[2,1]}) \ar[r]^{{i'}_*} \ar[d]^{j_D^*} & \Agl^1(\wt{\PP}(V_n)_{[2,1]}) \ar[d]^{j^*} \\
		\Agl^0(D_2) \ar[r]^{i^{2}_*} \ar[d]^{\del_D} & \Agl^1(\PP(V_n)_2) \ar[d]^{\del} \\
		\Agl^0(D'_1) \ar[r]^{i'^1_*} \ar[d]^{f_{D*}} & \Agl^1(\wt{\PP}(V_n)_1) \ar[d]^{f_*} \\
		\Agl^1(D'_{[2,1]}) \ar[r]^{i'_*}  & \Agl^2(\wt{\PP}(V_n)_{[2,1]}) }
\end{equation}
Observe that lemma \ref{lm:Ptilde2=P2} allowed us to identify in the diagram above $\Agl^0(D'_2)\simeq \Agl^0(D_2)$ and $\Agl^0(\wt{\PP}(V_n)_2)\simeq \Agl^0(\PP(V_n)_2)$.

\begin{lm}\label{cor: Qtilde is a projective bundle}
	We have $\Agl^0(\wh{Q}\sng_{[2,1]})\simeq \Agl^0(\PP^2)$ and $\Agl^1(\wh{Q}\sng_{[2,1]})\simeq \Agl^1(\PP^2)$.
\end{lm}
\begin{proof}
	Let $\wh{Q}\sng$ be the closed subscheme of $\bA(2,2)\times\PP^2$ defined as follows:
	\[ \wh{Q}\sng:=\left\{ (q,p) \text{ such that }q_x(p)=q_y(p)=q_z(p)=0 \right\}\]
	Observe that $\wh{Q}\sng\to\PP^2$ is a $\GLt$-equivariant vector subbundle of the (trivial) vector bundle $\bA(2,2)\times\PP^2\to\PP^2$. Therefore, for the homotopy invariance of the Chow groups with coefficients (see proposition \ref{pr:properties}.(6)) we get $\Agl^1(\wh{Q}\sng)\simeq\Agl^1(\PP^2)$.
	
	The scheme $\wh{Q}\sng_{[2,1]}$ is the complement in $\wh{Q}\sng$ of the zero section, which has codimension $3$. The localization exact sequence (see proposition \ref{pr:properties}.(4)) implies that, if $Z$ is a closed subscheme of codimension $>i$ of a scheme $X$, then the Chow group with coefficients of codimension $i$ of $X$ and $X\setminus Z$ coincide.
	
	Applying this to our case, we deduce that $\Agl^i(\wh{Q}\sng)\simeq\Agl^i(\wh{Q}\sng_{[2,1]})$ for $i=0,1$.
\end{proof}
\begin{lm}\label{cor: not zero iff not zero for E}
	The Chow group with coefficients $\Agl^1(\wt{\PP}(V_n)_{[2,1]})$ is concentrated in degree $0$. In particular, given an element $\beta$ of $\Agl^0(D_2)$ whose degree is greater than $0$, then $i^2_*\beta=0$ if and only if $i'^1_*(\del_D\beta)=0$.
\end{lm}
\begin{proof}
	Applying the projective bundle formula (proposition \ref{pr:properties}.(7)) we get:
	\[\Agl^1(\wt{\PP}(V_n)_{[2,1]})\simeq \Agl^1(\wh{Q}\sng_{[2,1]})\oplus \Agl^0(\wh{Q}\sng_{[2,1]})\cdot h\]
	We know from lemma \ref{cor: Qtilde is a projective bundle} that $\Agl^i(\wh{Q}\sng_{[2,1]})$ is concentrated in degree $0$ for $i=0,1$: this proves the first part of the lemma.
	
	Given an element $\beta$ in $\Agl^0(D_2)$ of degree greater than $0$, suppose that $i'^1_*(\del_D\beta)=0$. The commutativity of diagram \ref{diag tilde} implies that $\del(i^2_*\beta)=0$, and the exactness of the right vertical sequence tells us that $i^2_*\beta=j^*\gamma$. Observe that 
	\[0<\deg(\beta)=\deg(i^2_*\beta)=\deg(j^*\gamma)=\deg(\gamma) \]
	because pushforwards and pullback preserve the degree (see proposition \ref{pr:properties}.(2) and \ref{pr:properties}.(3)). The group $\Agl^1(\wt{\PP}(V_n)_{[2,1]})$ is concentrated in degree $0$, hence $\gamma=0$ and this concludes the proof.
\end{proof}

\begin{prop}\label{pr:i_* is zero in degree two}
	If $\beta$ is an element of degree $1$ in $\Agl^0(D_2)$, then $i^{2}_*\beta=0$.
\end{prop}
\begin{proof}
	By lemma \ref{cor: not zero iff not zero for E}, we can equivalently show that $i'^1_*(\del_D\beta)=0$. 
	Observe that $\del_D(\beta)$ has degree $0$. The degree $0$ part of $\Agl^0(D'_1)$ is isomorphic to $\CHgl^0(D'_1)_{\FF_2}$ by proposition \ref{pr:properties}.(1).
	Recall that $D'_1$ is the union of two irreducible components (see proposition \ref{pr:DtildeE has two irr comp}), hence we have:
	\[\CHgl^0(D'_1)_{\FF_2}=\CHgl^0(D'^1_1)_{\FF_2}\oplus \CHgl^0(D'^2_1)_{\FF_2}\]
	By exactness of the left vertical sequence of diagram \ref{diag tilde}, we have that $f_{D*}(\del_D\beta)=0$, thus $i'_*f_{D*}(\del_D\beta)=0$.
	If we write $\del_D\beta$ as $(n,m)$, then we are saying that $n[D'^1_1]+m[D'^2_1]=0$ in $\CHgl^1(\wt{\PP}(V_n)_{[2,1]})_{\FF_2}$. By corollary \ref{cor:cycle class Dtilde1 non zero} and lemma \ref{lm:cycle class of Dtilde2} we know that $[D'^1_1]\neq 0$ and $[D'^2_1]=0$, thus $n=0$.

	We have proved that $\del_D\beta=(0,m)$. This implies that $i'^1_*(\del_D\beta)=m[D'^2_1]$, which is equal to zero by lemma \ref{lm:cycle class of Dtilde2}.
\end{proof}
We are ready to prove the key lemma \ref{lm:key}.
\begin{proof}[Proof of key lemma \ref{lm:key}]
	We want to prove that:
	\[ i_*:\Agl^0(D_3)\longrightarrow \Agl^1(\PP(V_n)_3)\]
	is zero. As already observed, $\Agl^1(\PP(V_n)_3)$ is concentrated in degree $0$, $1$ and $2$, hence every element $\alpha$ in $\Agl^0(D_3)$ of degree $>2$ will be sent to $0$, because $i_*$ preserves the degree (see proposition \ref{pr:properties}.(1)).
	
	Suppose $\deg(\alpha)=0$. The morphism $i_*$ restricted to the degree $0$ part is equal to
	\[i_*:\CHgl^0(D_3)_{\FF_2}\to\CHgl^1(\PP(V_n)_3)_{\FF_2}\]
	by proposition \ref{pr:properties}.(1), and we know that this morphism is zero by lemma \ref{lm: i_sm is zero in degree 0}.
	
	Suppose $\deg(\alpha)=1$. Then proposition \ref{pr:i* is zero for elements of degree one} tells us that $i_*\alpha=0$. 
	
	The only case left is when $\deg(\alpha)=2$. By lemma \ref{cor: not zero iff not zero for nod} we have that $i_*\alpha=0$ if and only if $i^{2}_*(\del_D\alpha)=0$. 
	Observe that $\deg(\del_D\alpha)=1$, hence we can apply proposition \ref{pr:i_* is zero in degree two} to deduce that $i^{2}_*(\del_D\alpha)=0$. This finishes the proof.
\end{proof}

\end{document}